\documentclass{amsart}

\usepackage{thesis} 

\begin{document}

\title{\MakeUppercase{\rnervestitle}}
\date{\today}
\author{\textsc{Aaron Mazel-Gee}}

\begin{abstract}
We functorially associate to each relative $\infty$-category $(\R,\bW)$ a simplicial space $\NerveRezki(\R,\bW)$, called its \textit{Rezk nerve} (a straightforward generalization of Rezk's ``classification diagram'' construction for relative categories).  We prove the following \textit{local} and \textit{global} universal properties of this construction: (i) that the complete Segal space generated by the Rezk nerve $\NerveRezki(\R,\bW)$ is precisely the one corresponding to the localization $\loc{\R}{\bW}$; and (ii) that the Rezk nerve functor defines an equivalence $\loc{\RelCati}{\bW_\BarKan} \xra{\sim} \Cati$ from a localization of the $\infty$-category of relative $\infty$-categories to the $\infty$-category of $\infty$-categories.
\end{abstract}

\maketitle

\papernum{2}

\setcounter{tocdepth}{1}
\tableofcontents

\setcounter{section}{-1}

\section{Introduction}

\subsection{The Rezk nerve}

\boilerplateintro{\R}\footnote{For instance, even in the case that $\R$ is a one-object 1-category and we are only interested in its \textit{1-categorical} localization, i.e.\! the composite $\R \ra \loc{\R}{\bW} \ra \ho(\loc{\R}{\bW}) \simeq \R[\bW^{-1}]$ -- that is, in the case that we are interested in freely inverting certain elements of a monoid --, obtaining a concrete description is nevertheless an intractable (in fact, computationally undecidable) task, closely related to the so-called ``word problem'' for generators and relations in abstract algebra.}  To ameliorate this state of affairs, in this paper we provide a novel method of accessing this localization via Rezk's theory of \bit{complete Segal spaces}.

To describe this, let us first recall that the $\infty$-category $\CSS$ of complete Segal spaces participates in a diagram
\[ \begin{tikzcd}[column sep=1.5cm]
s\S
{\arrow[transform canvas={yshift=0.7ex}]{r}{\leftloc_\CSS}[swap, transform canvas={yshift=0.25ex}]{\scriptstyle \bot} \arrow[transform canvas={yshift=-0.7ex}, hookleftarrow]{r}[swap]{\forget_\CSS}}
& \CSS
{\arrow[transform canvas={yshift=0.7ex}]{r}{\Nervei^{-1}}[swap, transform canvas={yshift=0.05ex}]{\sim} \arrow[transform canvas={yshift=-0.7ex}, leftarrow]{r}[swap]{\Nervei}}
& \Cati .
\end{tikzcd} \]
That is, it sits as a reflective subcategory of the $\infty$-category $s\S$ of simplicial spaces, and it is equivalent to the $\infty$-category $\Cati$ of $\infty$-categories.  In particular, one can contemplate the complete Segal space (or equivalently, the $\infty$-category) \textit{generated} by an arbitrary simplicial space $Y$, much as one can contemplate the 1-category generated by an arbitrary simplicial set: this is encoded by the unit
\[ Y \xra{\eta} \leftloc_\CSS(Y) \]
of the adjunction (where we omit the inclusion functor $\forget_\CSS$ for brevity).

Now, given a relative $\infty$-category $(\R,\bW)$, its \bit{Rezk nerve} is a certain simplicial space
\[ \NerveRezki(\R,\bW) \in s\S \]
which ``wants to be'' the complete Segal space
\[ \Nervei(\loc{\R}{\bW}) \in \CSS \]
corresponding to its localization:
\begin{itemizesmall}
\item it admits canonical maps
\[ \Nervei(\R) \ra \NerveRezki(\R,\bW) \ra \Nervei(\loc{\R}{\bW}) , \]
and moreover
\item its construction manifestly dictates that for any $\infty$-category $\C$, the restriction map
\[ \hom_{s\S}(\NerveRezki(\R,\bW) , \Nervei(\C)) \ra \hom_{s\S}(\Nervei(\R),\Nervei(\C)) \simeq \hom_\Cati(\R,\C) \]
factors through the subspace of those functors $\R \ra \C$ sending all maps in $\bW \subset \R$ to equivalences in $\C$.
\end{itemizesmall}
Unfortunately, life is not quite so simple: the Rezk nerve is not generally a complete Segal space (or even a Segal space).\footnote{We provide sufficient conditions on $(\R,\bW)$ for its Rezk nerve $\NerveRezki(\R,\bW)$ to be a (complete) Segal space in \cite{MIC-hammocks}.}  Nevertheless, the second-best-possible thing is true.

\begin{thm*}[\ref{rezk nerve of a relative infty-category is initial}]
The above maps extend to a commutative diagram
\[ \begin{tikzcd}
\Nervei(\R) \arrow{r} \arrow{d}[swap]{\eta}[sloped, anchor=south]{\sim} & \NerveRezki(\R,\bW) \arrow{r} \arrow{d}{\eta} & \Nervei(\loc{\R}{\bW}) \arrow{d}{\eta}[sloped, anchor=north]{\sim} \\
\leftloc_\CSS(\Nervei(\R)) \arrow{r} & \leftloc_\CSS(\NerveRezki(\R,\bW)) \arrow{r}[swap]{\sim} & \leftloc_\CSS(\Nervei(\loc{\R}{\bW})) .
\end{tikzcd} \]
\end{thm*}

\noindent In other words, the complete Segal space generated by the Rezk nerve of $(\R,\bW)$ is precisely the one corresponding to its localization.

This theorem provides a \textit{local} universal property of the Rezk nerve: it asserts that the composite
\[ \RelCati \xra{\NerveRezki} s\S \xra{\leftloc_\CSS} \CSS \xra[\sim]{\Nervei^{-1}} \Cati \]
takes each relative $\infty$-category $(\R,\bW)$ to its localization $\loc{\R}{\bW}$.  However, it says nothing about the effect of this composite on \textit{morphisms} of relative $\infty$-categories.  To this end, we also prove the following.

\begin{thm*}[\ref{rezk nerve induces localization} and \ref{Rezk nerve computes localization}]
The above composite is canonically equivalent to the localization functor
\[ \RelCati \ra \Cati . \]
In particular, denoting by $\bW_\BarKan \subset \RelCati$ the subcategory of maps which it takes to equivalences, the above composite induces an equivalence
\[ \loc{\RelCati}{\bW_\BarKan} \xra{\sim} \Cati . \]
\end{thm*}

\noindent In other words, the Rezk nerve functor does indeed \textit{functorially} compute localizations of relative $\infty$-categories, and moreover the induced ``homotopy theory'' on the $\infty$-category $\RelCati$ of relative $\infty$-categories -- that is, the relative $\infty$-category structure $(\RelCati,\bW_\BarKan)$ that results therefrom -- gives a presentation of the $\infty$-category $\Cati$ of $\infty$-categories.  We therefore deem this result as capturing the \textit{global} universal property of the Rezk nerve.


\begin{rem}\label{rem intro infty-catl rezk nerve agrees with 1-catl rezk nerve}
The Rezk nerve functor is a close cousin of Rezk's ``classification diagram'' functor of \cite[3.3]{RezkCSS}; to emphasize the similarity, we denote the latter functor by
\[ \strrelcat \xra{\NerveRezk} ss\Set \]
and refer to it as the \textit{1-categorical Rezk nerve}.  In fact, as we explain in \cref{infty-catl rezk nerve agrees with 1-catl rezk nerve}, this is essentially just the restriction of the $\infty$-categorical Rezk nerve functor, in the sense that there is a canonical commutative diagram
\[ \begin{tikzcd}
\strrelcat \arrow{d} \arrow{r}{\NerveRezk} & s(s\Set) \arrow{r}{s(|{-}|)} & s\S \\
\RelCati \arrow[bend right=10]{rru}[swap, sloped, pos=0.4]{\NerveRezki}
\end{tikzcd} \]
in $\Cati$.  In \cref{BK proved global univ prop for relcats}, we use this observation to show that our global universal property of the $\infty$-categorical Rezk nerve can be seen as a generalization of work of Barwick--Kan.
\end{rem}

\subsection{Conventions}

\partofMIC

\tableofcodenames

\examplecodename

\citelurie \ \luriecodenames

\butinvt \ \seeappendix

\subsection{Outline}

We now provide a more detailed outline of the contents of this paper.

\begin{itemize}

\item In \cref{section rel-infty-cats and localizns}, we undertake a study of relative $\infty$-categories and their localizations.

\item In \cref{section CSSs}, we briefly review the theory of complete Segal spaces.

\item In \cref{section rezk nerve}, we introduce the Rezk nerve and state its local and global universal properties.  We give a proof of the global universal property which relies on the local one, but we defer the proof of the local one to \cref{section proof of nerve}.

\item In \cref{section proof of nerve}, we prove the local universal property of the Rezk nerve.  Though much of the proof is purely formal, at its heart it ultimately relies on some rather delicate model-categorical arguments.

\end{itemize}

\subsection{Acknowledgments}

We heartily thank Zhen Lin Low, Eric Peterson, Chris Schommer-Pries, and Mike Shulman for many (sometimes extremely extended) discussions regarding the material in this paper, particularly the proof of \cref{natural mono}.  It is also our pleasure to thank Katherine de Kleer for writing a Python script verifying the identities for the simplicial homotopies defined therein.\footnote{This script is readily available upon request.}  Lastly, we thank the NSF graduate research fellowship program (grant DGE-1106400) for its financial support during the time that this work was carried out.

\section{Relative $\infty$-categories and their localizations}\label{section rel-infty-cats and localizns}

Given an $\infty$-category and some chosen subset of its morphisms, we are interested in freely inverting those morphisms.  In order to codify these initial data, we introduce the following.

\begin{defn}\label{define rel infty-cat}
A \bit{relative $\infty$-category} is a pair $(\R,\bW)$ of an $\infty$-category $\R$ and a subcategory $\bW\subset \R$, called the subcategory of \bit{weak equivalences}, such that $\bW$ contains all the equivalences (and in particular, all the objects) in $\R$.  These form the evident $\infty$-category $\RelCati$.\footnote{To be precise, one can view $\RelCati \simeq \Fun^{\textup{surj mono}}([1],\Cati) \subset \Fun([1],\Cati)$ as the full subcategory on those functors selecting the inclusion of a surjective monomorphism.}  Weak equivalences will be denoted by the symbol $\we$.  Though we will of course write $\R$ for the $\infty$-category obtained by forgetting $\bW$, to ease notation we will also sometimes simply write $\R$ for the pair $(\R,\bW)$.  We write $\RelCat \subset \RelCati$ for the full subcategory on those relative $\infty$-categories $(\R,\bW)$ such that $\R \in \Cat \subset \Cati$.
\end{defn}

\begin{rem}\label{weaker defn of rel infty-cat}
As we are working invariantly, our \cref{define rel infty-cat} is not quite a generalization of the 1-category $\strrelcat$ of relative categories as given e.g.\! in \cite[3.1]{BK-relcats} or \cite[Definition 3.1]{LowMG}, an object of which is a \textit{strict} category $\R \in \strcat$ (see subitem \Cref{sspaces:section conventions}\cref{sspaces:work invariantly with qcats}\cref{sspaces:strict cats}) equipped with a wide subcategory $\bW \subset \R$ (i.e.\! one containing all the objects).  For emphasis, we will therefore sometimes refer to objects of $\strrelcat$ as \textit{strict relative categories}.

In addition to being the only meaningful variant in the invariant world, \cref{define rel infty-cat} allows for a clean and aesthetically appealing definition of localization, namely as a left adjoint (see \cref{define localization}).  In any case, as we are ultimately only interested in relative $\infty$-categories because we are interested in their localizations, this requirement is no real loss.

Despite these differences, there is an evident functor
\[ \strrelcat \ra \RelCat , \]
to which we will refer on occasion.
\end{rem}

\begin{notn}\label{notn for decorating relcat data}
In order to disambiguate our notation associated to various relative $\infty$-categories, we introduce the following conventions.
\begin{itemize}

\item When multiple relative $\infty$-categories are under discussion, we will sometimes decorate them for clarity.  For instance, we may write $(\R_1,\bW_1)$ and $(\R_2,\bW_2)$ to denote two arbitrary relative $\infty$-categories, or we may instead write $(\I,\bW_\I)$ and $(\J,\bW_\J)$.

\item Moreover, we will eventually study certain ``named'' relative $\infty$-categories; for example, there is a \textit{Barwick--Kan relative structure} on $\RelCati$ itself (see \cref{define BarKan rel str on RelCati}).  We will always subscript the subcategory of weak equivalences of such a relative $\infty$-category with (an abbreviation of) its name; for example, we will write $\bW_\BarKan \subset \RelCati$. We may also merely similarly subscript the ambient $\infty$-category to denote the relative $\infty$-category; for example, we will write $(\RelCati)_\BarKan = (\RelCati , \bW_\BarKan)$.

\item Finally, there will occasionally be two different $\infty$-categories with relative structures of the same name.  In such cases, if disambiguation is necessary we will additionally superscript the subcategory of weak equivalences with the name of the ambient $\infty$-category.  For instance, we would write $\bW^\RelCati_\BarKan \subset \RelCati$ to distinguish it from the subcategory $\bW^\strrelcat_\BarKan \subset \strrelcat$.

\end{itemize}
\end{notn}

We have the following fundamental source of examples of relative $\infty$-categories.

\begin{ex}\label{ex create relative infty-category}
If $\R \ra \C$ is any functor of $\infty$-categories, we can define a relative $\infty$-category $(\R,\bW)$ by declaring $\bW \subset \R$ to be the subcategory on those maps that are sent to equivalences in $\C$.  Note that $\bW \subset \R$ will automatically have the two-out-of-three property.
\end{ex}

\begin{defn}
In the situation of \cref{ex create relative infty-category}, we will say that the functor $\R \ra \C$ \bit{creates} the subcategory $\bW \subset \R$.
\end{defn}

We will make heavy use of the following construction.

\begin{notn}\label{define internal hom in relcats}
Given any $(\R_1,\bW_1) , (\R_2,\bW_2) \in \RelCati$, we define
\[ \left( \Fun(\R_1,\R_2)^\Rel , \Fun(\R_1,\R_2)^\bW \right) \in \RelCati \]
by setting
\[ \Fun(\R_1,\R_2)^\Rel \subset \Fun(\R_1,\R_2) \]
to be the full subcategory on those functors which send $\bW_1 \subset \R_1$ into $\bW_2 \subset \R_2$, and setting
\[ \Fun(\R_1,\R_2)^\bW \subset \Fun(\R_1,\R_2)^\Rel \]
to be the (generally non-full) subcategory on the natural weak equivalences.\footnote{If we consider $\RelCati \subset \Fun([1],\Cati)$, then $\Fun(\R_1,\R_2)^\Rel$ is simply the $\infty$-category of natural transformations.}  It is not hard to see that this defines an internal hom bifunctor for $(\RelCati,\times)$.
\end{notn}

It will be useful to have the following terminology.

\begin{defn}\label{trivial relative infty-category structures}
If $\C$ is any $\infty$-category, we call $(\C,\C^\simeq)$ the associated \bit{minimal relative $\infty$-category} and we call $(\C,\C)$ the associated \bit{maximal relative $\infty$-category}.  These define fully faithful inclusions
\[ \begin{tikzcd}[column sep=1.5cm]
\Cati \arrow[bend left, in=155]{r}{\min} \arrow[bend right, in=205]{r}[swap]{\max} & \RelCati \arrow{l}[pos=0.45, transform canvas={yshift=-0.6ex}]{\perp}[swap, pos=0.45, transform canvas={yshift=0.6ex}]{\perp}
\end{tikzcd} \]
which are respectively left and right adjoint to the forgetful functor $\RelCati \xra{\forget_\Rel} \Cati$ sending $(\R,\bW)$ to $\R$.  For $[n] \in \bD \subset \Cati$, we will use the abbreviation $[n]_\bW = \max([n])$, since these relative categories will appear quite often; correspondingly, we will also make the implicit identification $[n] = \min([n])$.
\end{defn}

We now come to our central object of interest.

\begin{defn}\label{define localization}
The functor $\min : \Cati \ra \RelCati$ also admits a left adjoint
\[ \RelCati \xra{\locL} \Cati , \]
which we refer to as the \bit{localization} functor on relative $\infty$-categories.  For a relative $\infty$-category $(\R,\bW) \in \RelCati$, we will often write $\loc{\R}{\bW} = \locL(\R,\bW)$; we only write $\locL$ since the notation $\loc{(-)}{(-)}$ is a bit unwieldy.  Explicitly, its value on $(\R,\bW) \in \RelCati$ can be obtained as the pushout
\[ \loc{\R}{\bW} \simeq \colim \left( \begin{tikzcd}
\bW \arrow{r} \arrow{d} & \R \\
\bW^\gpd
\end{tikzcd} \right) \]
in $\Cati$ (and the functor itself can be obtained by applying this construction in families).
\end{defn}

\begin{rem}
Using model categories, one can of course compute the pushout in $\Cati$ of \cref{define localization} by working in $s\Set_\Joyal$ (which is left proper), for instance after presenting the map $\bW \ra \bW^\gpd$ using the derived unit of the Quillen adjunction $\id : s\Set_\Joyal \adjarr s\Set_\KQ : \id$, i.e.\! after taking a fibrant replacement via a cofibration in $s\Set_\KQ$ of a quasicategory presenting $\bW$.  However, note that this derived unit can be quite difficult to describe in practice, and moreover the resulting pushout will generally still be very far from being a quasicategory.  Equally inexplicitly, one can also obtain a quasicategory presenting $\loc{\R}{\bW}$ by computing a fibrant replacement in the marked model structure of Proposition T.3.1.3.7 (i.e.\! in the specialization of the model structure given there to the case where the base is the terminal object $\pt_{s\Set}$).
\end{rem}

\begin{rem}\label{rem localization induces mono of spaces}
We will also use the term ``localization'' to refer to the canonical map $\R \ra \loc{\R}{\bW}$ in $\Cati$ satisfying the universal property that for any $\C \in \Cati$, the restriction
\[ \hom_{\Cati}(\loc{\R}{\bW},\C) \ra \hom_{\Cati}(\R,\C) \]
defines an equivalence onto the subspace
\[ \hom_{\RelCati}((\R,\bW),\min(\C)) \subset \hom_{\Cati}(\R,\C) \]
of those functors which take $\bW$ into $\C^\simeq$.\footnote{This map can be obtained either by applying $\RelCati \xra{\locL} \Cati$ to the counit $\min(\R) \ra (\R,\bW)$ of the adjunction $\min \adj \forget_\Rel$, or by applying $\RelCati \xra{\forget_\Rel} \Cati$ to the unit $(\R,\bW) \ra \min(\loc{\R}{\bW})$ of the adjunction $\locL \adj \min$.}  Thus, by definition the map $\R \ra \loc{\R}{\bW}$ is an epimorphism in $\Cati$.
\end{rem}

\begin{ex}\label{ex minimal localization}
The localization of a minimal relative $\infty$-category $\min(\C) = (\C,\C^\simeq)$ is simply the identity functor $\C \xra{\sim} \C$.
\end{ex}

\begin{ex}\label{ex maximal localization}
The localization of a maximal relative $\infty$-category $\max(\C) = (\C,\C)$ is the groupoid completion functor $\C \ra \C^\gpd$ (i.e.\! the component at $\C$ of the unit of the adjunction $(-)^\gpd : \Cati \adjarr \S : \forget_\S$).
\end{ex}

\begin{ex}\label{ex left localization qua free localization}
Given a left localization adjunction $\leftloc : \C \adjarr \leftloc\C : \forget$, if we define $\bW \subset \C$ to be created by $\C \xra{\leftloc} \leftloc\C$, then the localization of $(\C,\bW)$ is precisely $\C \xra{\leftloc} \leftloc\C$: that is, the functor $\C \xra{\leftloc} \leftloc\C$ induces an equivalence $\loc{\C}{\bW} \xra{\sim} \leftloc\C$, which is in fact inverse to the composite $\leftloc\C \xra{\forget} \C \ra \loc{\C}{\bW}$.  This follows from Proposition T.5.2.7.12, or alternatively from \cref{nat w.e. induces equivce betw fctrs} (see \cref{rem re-prove that a left localization is a free localization}).  Of course, a dual statement holds for right localization adjunctions.
\end{ex}

For an arbitrary relative $\infty$-category $(\R,\bW)$, note that the localization map $\R \ra \loc{\R}{\bW}$ might \textit{not} create the subcategory $\bW \subset \R$: there might be strictly more maps in $\R$ which are sent to equivalences in $\loc{\R}{\bW}$.  This leads us to the following notion.

\begin{defn}\label{define saturated}
A relative $\infty$-category $(\R,\bW)$ is called \bit{saturated} if the localization map $\R \ra \loc{\R}{\bW}$ creates the subcategory $\bW \subset \R$.
\end{defn}

\begin{rem}
If a relative $\infty$-category $(\R,\bW) \in \RelCati$ has its subcategory of weak equivalences $\bW \subset \R$ created by \textit{any} functor $\R \ra \C$, then $(\R,\bW)$ will automatically be saturated.  This is true by definition if $\bW \subset \R$ is created by the localization functor $\R \ra \loc{\R}{\bW}$.  More generally, if it is created by any other functor $\R \ra \C$, then in the canonical factorization
\[ \R \ra \loc{\R}{\bW} \ra \C , \]
the second functor will be conservative.  Hence, it will be also true that the subcategory $\bW \subset \R$ is created by the localization map $\R \ra \loc{\R}{\bW}$, which reduces us to the previous special case.
\end{rem}

Now, we will be using relative $\infty$-categories as ``presentations of $\infty$-categories'', namely of their localizations.  However, a map of relative $\infty$-categories may induce an equivalence on localizations without itself being an equivalence in $\RelCati$.  This leads us to the following notion.

\begin{defn}\label{define BarKan rel str on RelCati}
We define the subcategory $\bW_\BarKan \subset \RelCati$ of \bit{Barwick--Kan weak equivalences} to be created by the localization functor $\RelCati \xra{\locL} \Cati$.  We denote the resulting relative $\infty$-category by $(\RelCati)_\BarKan = (\RelCati , \bW_\BarKan) \in \RelCati$.
\end{defn}

The following result then justifies our usage of relative $\infty$-categories as ``presentations of $\infty$-categories''.

\begin{prop}\label{loc induces BK-equivce}
The functors in the left localization adjunction $\locL : \RelCati \adjarr \Cati : \min$ induce inverse equivalences
\[ \loc{\RelCati}{\bW_\BarKan} \simeq \Cati \]
in $\Cati$.
\end{prop}

\begin{proof}
This is a special case of \cref{ex left localization qua free localization}.
\end{proof}

We have the following strengthening of \cref{rem localization induces mono of spaces}.

\begin{prop}\label{localization induces full mono of infty-cats}
For any $(\R,\bW) \in \RelCati$ and any $\C \in \Cati$, the restriction
\[ \Fun(\loc{\R}{\bW},\C) \ra \Fun(\R,\C) \]
along the localization functor $\R \ra \loc{\R}{\bW}$ defines an equivalence onto the full subcategory of $\Fun(\R,\C)$ spanned by those functors which take $\bW$ into $\C^\simeq$.  
\end{prop}

\begin{proof}
We begin by observing that this functor is a monomorphism in $\Cati$: this is because we have a pullback diagram
\[ \begin{tikzcd}
\Fun( \loc{\R}{\bW} , \C) \arrow{r} \arrow{d} & \Fun(\bW^\gpd , \C) \arrow{d} \\
\Fun(\R,\C) \arrow{r} & \Fun(\bW , \C)
\end{tikzcd} \]
in $\Cati$ in which the right arrow is clearly a monomorphism, and monomorphisms are closed under pullback.  So in particular, this functor is the inclusion of a subcategory.  Then, to see that it is full, suppose we are given two functors $\loc{\R}{\bW} \rra \C$, considered as objects of $\Fun(\loc{\R}{\bW} , \C)$.  A natural transformation between their images in $\Fun(\R,\C)$ is given by a functor $[1] \times \R \ra \C$ which restricts to the the two composites $\R \ra \loc{\R}{\bW} \rra \C$ on the two objects $0 , 1 \in [1]$.  Since we already know that $\Fun(\loc{\R}{\bW} , \C) \subset \Fun(\R,\C)$ is the inclusion of a subcategory, it suffices to obtain an extension
\[ \begin{tikzcd}
{[1] \times \R} \arrow{r} \arrow{d} & \C \\
{[1] \times \loc{\R}{\bW}} \arrow[dashed]{ru}
\end{tikzcd} \]
in $\Cati$.  For this, consider the diagram
\[ \begin{tikzcd}
& \{0,1\} \times \bW \arrow{rr} \arrow{ld} \arrow{dd} & & \{0 , 1 \} \times \bW^\gpd \arrow{ld} \arrow{dd} \\
\{0,1\} \times \R \arrow[crossing over]{rr} \arrow{dd} & & \{0,1\} \times \loc{\R}{\bW} & & \C \\
& {[1] \times \bW} \arrow{rr} \arrow{ld} & & {[1] \times \bW^\gpd} \arrow{ld} \\
{[1] \times \R} \arrow{rr} \arrow[crossing over, controls={+(3,0.5) and +(-1.8,1)}, out=30]{rrrruu} & & {[1] \times \loc{\R}{\bW}} \arrow[leftarrow, crossing over]{uu}
\latearrow{crossing over}{2-3}{2-5}{}
\end{tikzcd} \]
in $\Cati$ containing and extending the above data.  The bottom square is a pushout since the functor $[1] \times - : \Cati \ra \Cati$ is a left adjoint, and the back square is a pushout by \cref{localization preserves finite products}.  Together, these observations guarantee the desired extension.
\end{proof}

\begin{rem}
\cref{localization induces full mono of infty-cats} implies that \cref{define localization} agrees with Definition A.1.3.4.1.
\end{rem}

We now make an easy observation regarding the localization functor, which is necessary for the argument of \cref{localization induces full mono of infty-cats} but will also be useful in its own right.

\begin{lem}\label{localization preserves finite products}
The localization functor $\locL : \RelCati \ra \Cati$ commutes with finite products.
\end{lem}

For the proof of \cref{localization preserves finite products}, it will be convenient to have the following notion.

\begin{defn}
Let $(\C,\otimes)$ be a closed symmetric monoidal $\infty$-category with internal hom bifunctor
\[ \C^{op} \times \C \xra{\enrhom_\C(-,-)} \C . \]
A collection of objects $I$ of $\C$ is called an \bit{exponential ideal} if we have $\enrhom_\C(Y,Z) \in I$ for any $Y \in \C$ and any $Z \in I$.  We will use this same terminology to refer to a full subcategory $\D \subset \C$ whose objects form an exponential ideal.
\end{defn}

The following straightforward result explains why we are interested in exponential ideals.

\begin{lem}\label{exp ideals and monoidal products}
Suppose that $(\C,\otimes)$ is a closed symmetric monoidal $\infty$-category, and let $\leftloc : \C \adjarr \leftloc\C : \forget$ be a left localization with unit map $\id_\C \xra{\eta} \leftloc$ in $\Fun(\C,\C)$ (where we implicitly consider $\leftloc\C \subset \C$).  Then, the full subcategory $\leftloc\C \subset \C$ is an exponential ideal if and only if the natural map $\leftloc(\eta \otimes \eta)$ is an equivalence in $\Fun(\C \times \C , \C)$ (i.e.\! we have
\[ \leftloc(Y \otimes Z) \xra{\sim} \leftloc(\leftloc(Y) \otimes \leftloc(Z)) \]
in $\leftloc\C$ for all $Y,Z \in \C$).  In particular, if $\leftloc\C$ is closed under the monoidal structure, then $\leftloc\C \subset \C$ is an exponential ideal if and only if
\[ \leftloc(Y \otimes Z) \simeq \leftloc(Y) \otimes \leftloc(Z) \]
in $\leftloc\C$ for all $Y , Z \in \C$.
\end{lem}

\begin{proof}
Suppose that $\leftloc\C \subset \C$ is an exponential ideal.  Then, for any $Y,Z \in \C$ and any test object $W \in \leftloc\C$, we have the string of natural equivalences
\begin{align*}
\hom_\C(\leftloc(Y \otimes Z) , W)
& \simeq \hom_\C(Y \otimes Z, W)
\simeq \hom_\C(Y , \enrhom_\C(Z,W))
\simeq \hom_\C(\leftloc(Y),\enrhom_\C(Z,W))
\\
& \simeq \hom_\C(\leftloc(Y) \otimes Z , W)
\simeq \hom_\C(Z \otimes \leftloc(Y) , W)
\simeq \hom_\C(Z , \enrhom_\C(\leftloc(Y),W))
\\
& \simeq \hom_\C(\leftloc(Z) , \enrhom_\C(\leftloc(Y),W))
\simeq \hom_\C(\leftloc(Z) \otimes \leftloc(Y) , W)
\\
& \simeq \hom_\C(\leftloc(Y) \otimes \leftloc(Z) , W)
\simeq \hom_\C(\leftloc(\leftloc(Y) \otimes \leftloc(Z)) , W) .
\end{align*}
Hence, we have an equivalence $\leftloc(Y \otimes Z) \simeq \leftloc(\leftloc(Y) \otimes \leftloc(Z))$ by the Yoneda lemma applied to the $\infty$-category $\leftloc\C$ (and it is straightforward to check that this equivalence is indeed induced by the specified map).  So $\leftloc(\eta \otimes \eta)$ is an equivalence in $\Fun(\C \times \C, \C)$, as desired.

On the other hand, suppose that $\leftloc(Y \otimes Z) \xra{\sim}\leftloc(\leftloc(Y) \otimes \leftloc(Z))$ for all $Y, Z \in \C$.  Then, we have the string of natural equivalences
\begin{align*}
\hom_\C(Y , \enrhom_\C(Z,W))
& \simeq \hom_\C(Y \otimes Z , W)
\simeq \hom_\C(\leftloc(Y \otimes Z),W)
\simeq \hom_\C(\leftloc(\leftloc(Y) \otimes \leftloc(Z)) , W)
\\
& \simeq \hom_\C(\leftloc(Y) \otimes \leftloc(Z) , W)
\simeq \hom_\C(\leftloc(\leftloc(Y)) \otimes \leftloc(Z), W )
\\
& \simeq \hom_\C(\leftloc(\leftloc(\leftloc(Y)) \otimes \leftloc(Z)) , W)
\simeq \hom_\C(\leftloc(\leftloc(Y) \otimes Z) , W)
\\
& \simeq \hom_\C(\leftloc(Y) \otimes Z , W)
\simeq \hom_\C(\leftloc(Y) , \enrhom_\C(Z,W)) .
\end{align*}
Hence, for any map $Y \ra Y'$ in $\C$ which localizes to an equivalence $\leftloc(Y) \xra{\sim} \leftloc(Y')$ in $\leftloc\C \subset \C$, we obtain an equivalence $\hom_\C(Y,\enrhom_\C(Z,W)) \xla{\sim} \hom_\C(Y',\enrhom_\C(Z,W))$.  It follows that the object $\enrhom_\C(Z,W) \in \C$ is local with respect to the left localization, i.e.\! that in fact $\enrhom_\C(Z,W) \in \leftloc\C \subset \C$.  So $\leftloc\C \subset \C$ is an exponential ideal.
\end{proof}

With \cref{exp ideals and monoidal products} in hand, we now proceed to prove \cref{localization preserves finite products}.

\begin{proof}[Proof of \cref{localization preserves finite products}]
The right adjoint $\min : \Cati \ra \RelCati$ induces an equivalence onto the full subcategory of minimal relative $\infty$-categories.  It is easy to see that this is an exponential ideal in $(\RelCati,\times)$, and so the result follows from \cref{exp ideals and monoidal products}.
\end{proof}

The following useful construction relies on \cref{localization preserves finite products}.

\begin{rem}\label{map from rel fctr cat to fctrs betw localizns}
Let $(\R_1,\bW_1) , (\R_2,\bW_2) \in \RelCati$.  Then the identity map
\[ \left( \Fun(\R_1,\R_2)^\Rel , \Fun(\R_1,\R_2)^\bW \right) \ra \left( \Fun(\R_1,\R_2)^\Rel , \Fun(\R_1,\R_2)^\bW \right) \]
is adjoint to an evaluation map
\[ (\R_1,\bW_1) \times \left( \Fun(\R_1,\R_2)^\Rel , \Fun(\R_1,\R_2)^\bW \right) \ra (\R_2,\bW_2) . \]
By \cref{localization preserves finite products}, applying the localization functor $\RelCati \xra{\locL} \Cati$ yields a map
\[ \loc{\R_1}{\bW_1} \times \loc{\Fun(\R_1,\R_2)^\Rel}{\left( \Fun(\R_1,\R_2)^\bW \right)} \ra \loc{\R_2}{\bW_2} , \]
which is itself adjoint to a canonical map
\[ \loc{\Fun(\R_1,\R_2)^\Rel}{\left( \Fun(\R_1,\R_2)^\bW \right)} \ra \Fun(\loc{\R_1}{\bW_1} , \loc{\R_2}{\bW_2} ). \]
In particular, precomposing with the localization map for the internal hom-object yields a canonical map
\[ \Fun(\R_1,\R_2)^\Rel \ra \Fun(\loc{\R_1}{\bW_1} , \loc{\R_2}{\bW_2} ) . \]
\end{rem}

\cref{localization preserves finite products} also allows us to prove the following result, which will be useful later and which gives a sense of the interplay between relative $\infty$-categories and their localizations.

\begin{lem}\label{nat w.e. induces equivce betw fctrs}
Given any $(\R_1\bW_1),(\R_2,\bW_2) \in \RelCati$ and any pair of maps $\R_1 \rra \R_2$ in $\RelCati$, a natural weak equivalence between them induces an equivalence between their induced functors $\loc{\R_1}{\bW_1} \rra \loc{\R_2}{\bW_2}$ in $\Cati$.
\end{lem}

\begin{proof}
A natural weak equivalence corresponds to a map $[1]_\bW \times \R_1 \ra \R_2$ in $\RelCati$.  By \cref{localization preserves finite products} (and \cref{ex maximal localization}), this gives rise to a map $[1]^\gpd \times \loc{\R_1}{\bW_1} \ra \loc{\R_2}{\bW_2}$ in $\Cati$, which precisely selects the desired equivalence.
\end{proof}

\begin{rem}\label{rem re-prove that a left localization is a free localization}
\cref{nat w.e. induces equivce betw fctrs} allows for a simple proof of Proposition T.5.2.7.12, that a left localization is in particular a free localization.  Indeed, given a left localization adjunction $\leftloc : \C \adjarr \leftloc\C : \forget$, write $\bW \subset \C$ for the subcategory created by the functor $\leftloc : \C \ra \leftloc\C$.  Then, this adjunction gives rise to a pair of maps $(\C,\bW) \xra{\leftloc} \min(\leftloc\C)$ and $\min(\leftloc\C) \xra{\forget} (\C,\bW)$ in $\RelCati$.  Moreover, the composite
\[ \min(\leftloc\C) \xra{\forget} (\C,\bW) \xra{\leftloc} \min(\leftloc\C) \]
is an equivalence, while the composite
\[ (\C,\bW) \xra{\leftloc} \min(\leftloc\C) \xra{\forget} (\C,\bW) \]
is connected to $\id_{(\C,\bW)}$ by the unit of the natural transformation, which is a componentwise weak equivalence (since for any $Y \in \C$, applying $\C \xra{\leftloc} \leftloc\C$ to the map $Y \ra \leftloc(Y)$ gives an equivalence $\leftloc(Y) \xra{\sim} \leftloc(\leftloc(Y))$).  Hence, it follows that these functors induce inverse equivalences $\loc{\C}{\bW} \simeq \leftloc\C$.  (From here, one can obtain the actual statement of Proposition T.5.2.7.12 by appealing to \cref{localization induces full mono of infty-cats}.)
\end{rem}

\cref{nat w.e. induces equivce betw fctrs} also has the following special case which will be useful to us.

\begin{lem}\label{nat trans induces equivce betw maps on gpd-complns}
Given any $\C,\D \in \Cati$ and any pair of maps $\C \rra \D$, a natural transformation between them induces an equivalence between the induced maps $\C^\gpd \rra \D^\gpd$ in $\S$.
\end{lem}

\begin{proof}
In light of \cref{ex maximal localization}, this follows from applying \cref{nat w.e. induces equivce betw fctrs} in the special case that $(\R_1,\bW_1) = \max(\C)$ and $(\R_2,\bW_2) = \max(\D)$.
\end{proof}

\begin{rem}
\cref{nat trans induces equivce betw maps on gpd-complns} can also be seen as following from applying \cref{exp ideals and monoidal products} to the left localization $(-)^\gpd : \Cati \adjarr \S : \forget_\S$.  Namely, since the full subcategory $\S \subset \Cati$ is an exponential ideal for $(\Cati,\times)$, then the left adjoint $(-)^\gpd : \Cati \ra \S$ commutes with finite products, and hence a natural transformation $[1] \times \C \ra \D$ gives rise to a map $([1] \times \C)^\gpd \simeq [1]^\gpd \times \C^\gpd \ra \D^\gpd$ which selects the desired equivalence in $\hom_\S(\C^\gpd,\D^\gpd)$.
\end{rem}

In turn, \cref{nat trans induces equivce betw maps on gpd-complns} has the following useful further special case.

\begin{cor}\label{adjns induce equivces on gpd-complns}
An adjunction $F : \C \adjarr \D : G$ induces inverse equivalences $F^\gpd : \C^\gpd \xra{\sim} \D^\gpd$ and $\C^\gpd \xla{\sim} \D^\gpd : G^\gpd$ in $\S$.
\end{cor}

\begin{proof}
The adjunction $F \adj G$ has unit and counit natural transformations $\id_\C \ra G \circ F$  and $F \circ G \ra \id_\D$, and so the claim follows from \cref{nat trans induces equivce betw maps on gpd-complns}.
\end{proof}

We note the following interaction between taking localizations and taking homotopy categories.

\begin{rem}\label{loc and ho commute}
Observe that the composite left adjoint
\[ \RelCati \xra{(\ho(-),\ho(-))} \RelCat \xra{(-)[(-)^{-1}]} \Cat \]
coincides with the composite left adjoint
\[ \RelCati \xra{\loc{(-)}{(-)}} \Cati \xra{\ho} \Cat , \]
since they share a right adjoint
\[ \RelCati \hookla \RelCat \xla{\min} \Cat . \]
Hence, for any $(\R,\bW) \in \RelCati$ we have a natural equivalence
\[ \ho(\loc{\R}{\bW}) \xra{\sim} \ho(\R)[\ho(\bW)^{-1}] \]
in $\Cat \subset \Cati$.
\end{rem}

We end this section with the following observation (which partly echoes \cref{sspaces:model structure on homotopy category}).

\begin{rem}
Suppose $(\R,\bW)$ is a relative $\infty$-category.  Then $(\ho(\R),\ho(\bW))$ is a relative category (so is in particular a relative $\infty$-category).  However, its localization $\loc{\ho(\R)}{\ho(\bW)}$ need not recover $\loc{\R}{\bW}$.  This is for the same reason as such facts always are, namely that we lose coherence data when we pass from $\R$ to $\ho(\R)$.  (Commutative diagrams in $\ho(\R)$ need not come from commutative diagrams in $\R$, and when they do they might do so in multiple, inequivalent ways.)  An explicit counterexample is provided by the minimal relative $\infty$-category $(\R,\bW) = (\R,\R^\simeq)$: then
\[ \ho(\bW) \simeq \ho(\R^\simeq) \simeq \ho(\R)^\simeq \subset \ho(\R) \]
since the equivalences in $\R$ are created by $\R \ra \ho(\R)$, and hence $\loc{\ho(\R)}{\ho(\bW)} \simeq \ho(\R)$ (while of course $\loc{\R}{\bW} \simeq \R$).  One might therefore refer to the $\infty$-category $\loc{\ho(\R)}{\ho(\bW)}$ as an ``exotic enrichment'' of the homotopy category $\ho(\loc{\R}{\bW})$.
\end{rem}

\section{Complete Segal spaces}\label{section CSSs}

We now give an extremely brief review of the theory of complete Segal spaces.  This section exists more-or-less solely to fix notation; we refer the reader seeking a more thorough discussion either to the original paper \cite{RezkCSS} (which uses model categories) or to \cite[\sec 1]{LurieGoo} (which uses $\infty$-categories).

Let us write $\bD \xra{[\bullet]} \Cat$ for the standard cosimplicial category.  Then, recall that the \textit{nerve} of a category $\C$ is the simplicial set $\Nerve(\C)_\bullet = \hom^\lw_\Cat([\bullet],\C)$.  This defines a fully faithful embedding $\Nerve : \Cat \ra s\Set$, with image those simplicial sets which admit \textit{unique} lifts for the inner horn inclusions $\{ \Lambda^n_i \ra \Delta^n \}_{0 < i < n \geq 0}$.  In fact, this functor is a right adjoint.

The situation with $\infty$-categories is completely analogous.

\begin{defn}\label{define CSSs}
The ($\infty$-categorical) \bit{nerve} of an $\infty$-category $\C$ is the simplicial space
\[ \Nervei(\C)_\bullet = \hom^\lw_{\Cati}([\bullet],\C) , \]
i.e.\! the composite
\[ \bD^{op} \xra{[\bullet]^{op}} \Cat^{op} \hookra (\Cati)^{op} \xra{\hom_{\Cati}(-,\C)} \S . \]
This defines a fully faithful embedding $\Nervei : \Cati \hookra s\S$, with image the full subcategory $\CSS \subset s\S$ of \bit{complete Segal spaces}, i.e.\! those simplicial spaces satisfying the \textit{Segal condition} and the \textit{completeness condition}.  This inclusion fits into a left localization adjunction $\leftloc_\CSS : s\S \adjarr \CSS : \forget_\CSS$.  Hence, we obtain an equivalence
\[ \Cati \xra[\sim]{\Nervei} \CSS , \]
whose inverse
\[ \CSS \xra[\sim]{\Nervei^{-1}} \Cati \]
takes an object $Y_\bullet \in \CSS$ to the coend
\[ \int^{[n] \in \bD} Y_n \times [n] \]
in $\Cati$.  (These claims respectively follow from Proposition A.A.7.10], \cite[Theorem 4.12]{JT}, \cite[Theorem 7.2]{RezkCSS}, and \cite[Theorem 4.12]{JT} again.)  This equivalence identifies subcategory $\S \subset \Cati$ with the subcategory of \textit{constant} simplicial spaces (which are automatically complete Segal spaces).
\end{defn}

\begin{rem}\label{extract hom-spaces from a CSS}
Complete Segal spaces provide an extremely efficient way of computing the hom-spaces in an $\infty$-category: if $x,y \in \C$, then there is a natural equivalence
\[ \hom_\C(x,y) \simeq \lim \left( \begin{tikzcd}
 & \Nervei(\C)_1 \arrow{d}{(s,t)} \\
 \pt_\S \arrow{r}[swap]{(x,y)} & \Nervei(\C)_0 \times \Nervei(\C)_0
\end{tikzcd} \right) \]
in $\S$, where we use the notations $s = \delta_1$ and $t = \delta_0$ to emphasize the roles that these two face maps play in this theory.  (Note that $\Nervei(\C)_0 = \hom_\Cati([0],\C) \simeq \C^\simeq$ is simply the maximal subgroupoid of $\C$, while $\Nervei(\C)_1 = \hom_\Cati([1],\C) \simeq \Fun([1],\C)^\simeq$ is the space morphisms in $\C$.)
\end{rem}

\begin{rem}\label{op on CSS}
There is a canonical involution $\bD \xra{\sim} \bD$ in $\strcat$, which is the identity on objects but acts on morphisms by ``reversing the coordinates'': a map $[m] \xra{\varphi} [n]$ is taken to the map
\[ [m] \xra{i \mapsto (n - \varphi(m-i))} [n] . \]
Taking opposites, this induces an involution $\bD^{op} \xra{\sim} \bD^{op}$, which in turn induces an involution of $s\S = \Fun(\bD^{op} , \S)$ by precomposition.  Unwinding the definitions, we see that this involution $s\S \xra{\sim} s\S$ restricts to an involution $\CSS \xra{\sim} \CSS$ which corresponds to the involution $(-)^{op} : \Cati \xra{\sim} \Cati$.
\end{rem}

For future use, we record the following observation.

\begin{prop}\label{groupoid-completion of CSSs}
The diagram
\[ \begin{tikzcd}[column sep=2cm, row sep=1.5cm]
s\S
\arrow[transform canvas={yshift=0.7ex}]{r}{\leftloc_\CSS}[swap, transform canvas={yshift=0.2ex}]{\scriptstyle \bot} \arrow[transform canvas={yshift=-0.7ex}, hookleftarrow]{r}[swap]{\forget_\CSS} \arrow{rd}[swap, sloped, pos=0.6]{|{-}|}
& \CSS 
{\arrow[transform canvas={yshift=0.7ex}]{r}{\Nervei^{-1}}[swap, transform canvas={yshift=0.05ex}]{\sim} \arrow[transform canvas={yshift=-0.7ex}, leftarrow]{r}[swap]{\Nervei}}
& \Cati \arrow{ld}[sloped, pos=0.6]{(-)^\gpd} \\
& \S
\end{tikzcd} \]
commutes: that is,
\begin{itemizesmall}
\item geometric realization of complete Segal spaces models groupoid completion of $\infty$-categories, and
\item for any $Y \in s\S$, the localization map $Y \ra \leftloc_\CSS(Y)$ becomes an equivalence upon geometric realization.
\end{itemizesmall}
\end{prop}

\begin{proof}
For the first claim, note that the functor $(-)^\gpd : \Cati \ra \S$ is a left localization, and the composite
\[ \S \xhookra{\forget_\S} \Cati \xra[\sim]{\Nervei} \CSS \xhookra{\forget_\CSS} s\S \]
agrees with the functor $\const : \S \ra s\S$.  Hence, the equivalence
\[ |{-}| \circ \forget_\CSS \circ \Nervei \simeq (-)^\gpd \]
in $\Fun(\Cati,\S)$ follows from the uniqueness of left adjoints.

For the second claim, note that the reflective inclusion $\const : \S \hookra s\S$ factors through the reflective inclusion $\forget_\CSS : \CSS \hookra s\S$.  Hence, the factorization $\S \hookra \CSS$ is also a reflective inclusion.  The equivalence
\[ |{-}| \simeq |{-}| \circ \forget_\CSS \circ \leftloc_\CSS \]
in $\Fun(s\S,\S)$ now also follows from the uniqueness of left adjoints.
\end{proof}

\begin{rem}
We may interpret \cref{groupoid-completion of CSSs} as saying that, while a simplicial space $Y \in s\S$ can be thought of as generating an $\infty$-category (namely the one corresponding to $\leftloc_\CSS(Y) \in \CSS$), we can already directly extract its groupoid completion from $Y$ itself.  This is analogous to the fact that an arbitrary simplicial set can be thought of as generating a quasicategory via fibrant replacement in $s\Set_\Joyal$, and the replacement map lies in $\bW_\Joyal \subset \bW_\KQ$ (i.e.\! it induces an equivalence on geometric realizations).
\end{rem}

\begin{rem}\label{nerve vs nervei}
Given a strict category $\C \in \strcat$, the maps $\hom_\strcatsup([n],\C) \ra \hom_\Cati([n],\C)$ from hom-sets to hom-spaces collect into a map
\[ \Nerve(\C) \ra \Nervei(\C) \]
in $s\S$; in turn, these maps assemble into a natural transformation $\Nerve \ra \Nervei$ in $\Fun(\strcat,s\S)$.  This map will be an equivalence in $s\S$ if and only if $\C$ is \textit{gaunt}: while the nerve $\Nerve(\C) \in s\Set \subset s\S$ is always a Segal space, it only satisfies the completeness condition when every isomorphism in $\C$ is actually an identity map.\footnote{Note that the Segal condition in $s\Set$ can be equivalently checked in $s\S$ since the inclusion $s\Set \subset s\S$ is a right adjoint.}  However, by \cite[Remark 7.8]{RezkCSS}, the above map induces an equivalence
\[ \leftloc_\CSS(\Nerve(\C)) \xra{\sim} \leftloc_\CSS(\Nervei(\C)) \simeq \Nervei(\C) \]
in $\CSS \subset s\S$.  In particular, it therefore follows from \cref{groupoid-completion of CSSs} that it also induces an equivalence
\[ |\Nerve(\C)| \xra{\sim} |\Nervei(\C)| \]
in $\S$.
\end{rem}

\section{The Rezk nerve}\label{section rezk nerve}

Recall that the \textit{localization} of a relative $\infty$-category $(\R,\bW)$ is the initial $\infty$-category $\loc{\R}{\bW}$ equipped with a functor from $\R$ which sends the subcategory $\bW \subset \R$ of weak equivalences to equivalences.  Meanwhile, given an arbitrary $\infty$-category $\C$, observe that the $n\th$ space of its nerve can be considered as
\[ \Nervei(\C)_n = \hom_{\Cati}([n],\C) \simeq \Fun( [n] , \C)^\simeq \subset \Fun([n],\C) , \]
the subcategory of $\Fun([n],\C)$ whose morphisms are the natural \textit{equivalences}.  Putting these two facts together, one is led to suspect that the $n\th$ space of the nerve $\Nervei(\loc{\R}{\bW})_\bullet$ should somehow contain the subcategory
\[ \Fun([n],\R)^\bW \subset \Fun([n],\R) \]
of $\Fun([n],\R)$ whose morphisms are the natural \textit{weak equivalences}.  Of course, this will not generally form a space, but will instead be an $\infty$-category.  On the other hand, there is a universal choice for a space admitting a map from this $\infty$-category, namely its groupoid completion.  We are thus naturally led to make the following construction, a direct generalization of the ``classification diagram'' construction for relative categories defined in \cite[3.3]{RezkCSS}.

\begin{defn}\label{define infty-categorical rezk pre-nerve and rezk nerve}
Given a relative $\infty$-category $(\R,\bW)$, its ($\infty$-categorical) \bit{Rezk pre-nerve} is the simplicial $\infty$-category
\[ \preNerveRezki(\R,\bW)_\bullet = \Fun^\lw([\bullet],\R)^\bW ,\]
i.e.\! the composite
\[ \bD^{op} \xra{[\bullet]^{op}} \Cat^{op} \hookra (\Cati)^{op} \xra{\min^{op}} (\RelCati)^{op} \xra{\Fun(-,\R)^\bW} \Cati . \]
This defines a functor
\[ \RelCati \xra{\preNerveRezki} s\Cati . \]
Then, the ($\infty$-categorical) \bit{Rezk nerve} functor
\[ \RelCati \xra{\NerveRezki} s\S \]
is given by the composite
\[ \RelCati \xra{\preNerveRezki} s\Cati \xra{s(-)^\gpd} s\S . \]
\end{defn}

\begin{rem}\label{infty-catl rezk nerve agrees with 1-catl rezk nerve}
Recall that Rezk's ``classification diagram'' construction of \cite[3.3]{RezkCSS}, which we will denote by
\[ \strrelcat \xra{\NerveRezk} s(s\Set) \]
and refer to as the \textit{1-categorical Rezk nerve} functor, is given by the formula
\[ \NerveRezk(\R,\bW)_\bullet = \Nerve \left( \strfun^\lw([\bullet],\R)^\bW \right) . \]
Of course, we would like to think of this as a simplicial space using the model category $s(s\Set_\KQ)_\Reedy$.  Indeed, combining \cref{groupoid-completion of CSSs} and \cref{nerve vs nervei}, we obtain a canonical commutative diagram
\[ \begin{tikzcd}
\strrelcat \arrow{d} \arrow{r}{\NerveRezk} & s(s\Set) \arrow{r}{s(|{-}|)} & s\S \\
\RelCati \arrow[bend right=10]{rru}[swap, sloped, pos=0.4]{\NerveRezki}
\end{tikzcd} \]
in $\Cati$; in fact, this even refines to a canonical commutative diagram
\[ \begin{tikzcd}
\strrelcat \arrow[hook]{d} \arrow{r}{\NerveRezk} & s(s\Set) \arrow{r} & s\Cati \arrow{r}{s(-)^\gpd} & s\S \\
\RelCati \arrow{rru}[swap, sloped, pos=0.35]{\preNerveRezki} \arrow[bend right=15]{rrru}[swap, sloped, pos=0.45]{\NerveRezki}
\end{tikzcd} \]
in $\Cati$ (in which the functor $s(s\Set) \ra s\Cati$ is obtained by applying $s(-) = \Fun(\bD^{op},-)$ to the localization $s\Set \ra \loc{s\Set}{\bW_\Joyal} \simeq \Cati$).  Thus, at least as far as homotopical content is concerned, the $\infty$-categorical Rezk nerve functor strictly generalizes its 1-categorical counterpart.
\end{rem}

\begin{rem}\label{rezk nerve for marked qcats}
In turn, the 1-categorical Rezk nerve functor of \cref{infty-catl rezk nerve agrees with 1-catl rezk nerve} suggests a similar model-dependent definition of a Rezk nerve functor for ``marked quasicategories'' (once again landing in $ss\Set$).  In fact, as the first step in the proof of \cref{natural mono}, we will show that this construction is a model-categorical presentation
\begin{itemizesmall}
\item of the $\infty$-categorical Rezk nerve when considered in $s(s\Set_\KQ)_\Reedy$, and in fact
\item of the $\infty$-categorical Rezk pre-nerve when considered in $s(s\Set_\Joyal)_\Reedy$.
\end{itemizesmall}
\end{rem}

\begin{rem}\label{pre-nerve as sCSS}
We have the following slight reformulation of \cref{define infty-categorical rezk pre-nerve and rezk nerve}: in view of \cref{groupoid-completion of CSSs}, the Rezk nerve functor can also be described as a composite
\[ \RelCati \xra{\preNerveRezki} s\Cati \simeq s\CSS \xhookra{s(\forget_\CSS)} s(s\S) \xra{s(|{-}|)} s\S . \]
Note that the composite functor $\RelCati \ra s(s\S)$ is a right adjoint, whose left adjoint is the left Kan extension
\[ \begin{tikzcd}[column sep=5cm]
\bD \times \bD \arrow{r}{\mxm = (([m],[n]) \mapsto [m] \times [n]_\bW)} \arrow{d}[swap]{\Yo} & \RelCati \\
s(s\S) \arrow[dashed]{ru}[sloped, swap, pos=0.35]{\Yo_!(\mxm)}
\end{tikzcd} \]
along the Yoneda embedding, where we write $\mxm$ for the upper ``$\min \times \max$'' functor for brevity.  On the other hand, the functor $s(|{-}|) : s(s\S) \ra s\S$ is a left adjoint.  Hence, as the Rezk nerve functor is the composite of a right adjoint followed by a left adjoint, understanding its behavior in general is a rather difficult task.  (In fact, it follows that $\preNerveRezki : \RelCati \ra s\Cati$ is also a right adjoint, while $s(-)^\gpd : s\Cati \ra s\S$ is of course also a left adjoint.)
\end{rem}

We have the following identifications of the Rezk nerves of minimal and maximal relative $\infty$-categories: in both of these extremal cases, the Rezk nerve does indeed compute the localization.

\begin{prop}\label{rezk nerves of min and max}
The Rezk nerve functor acts on the full subcategories of $\RelCati$ spanned by the minimal and maximal relative $\infty$-categories (both of which can be indentified with $\Cati$) according to the canonical commutative diagram
\[ \begin{tikzcd}
\Cati \arrow[hook]{r}{\min} \arrow{d}[swap]{\Nervei}[sloped, anchor=south]{\sim} & \RelCati \arrow[hookleftarrow]{r}{\max} \arrow{d}{\NerveRezki} & \Cati \arrow{d}{(-)^\gpd} \\
\CSS \arrow[hook]{r}[swap]{\forget_\CSS} & s\S \arrow[hookleftarrow]{r}[swap]{\const} & \S
\end{tikzcd} \]
in $\Cati$.
\end{prop}

\begin{proof}
To see that the left square commutes, given any $\C \in \Cati$ we compute that
\[ \preNerveRezki(\min(\C))_n = \Fun([n],\min(\C))^\bW \simeq \Fun([n],\C)^\simeq \simeq \hom_\Cati([n],\C) = \Nervei(\C)_n \]
(in a way compatible with the evident simplicial structure maps on both sides), i.e.\! we even have a canonical equivalence
\[ \preNerveRezki(\min(\C))_\bullet \simeq \Nervei(\C)_\bullet \]
in $s\Cati$.  As $s(-)^\gpd : s \Cati \adjarr s\S : s(\forget_\S)$ is a left localization adjunction, it follows that we also have a canonical equivalence
\[ \NerveRezki(\min(\C))_\bullet \simeq \Nervei(\C)_\bullet \]
in $s\S$.

To see that the right square commutes, given any $\C \in \Cati$ we first compute that
\[ \preNerveRezki(\max(\C))_n = \Fun([n],\max(\C))^\bW \simeq \Fun([n],\C) . \]
Moreover, note that every face-then-degeneracy composite
\[ \Fun([n],\C) \xra{\delta_i} \Fun([n-1],\C) \xra{\sigma_j} \Fun([n],\C) \]
admits a natural transformation either to or from $\id_{\Fun([n],\C)}$ (depending on $i$ and $j$).\footnote{We refer the reader to \cref{hammocks:nat w.e. induces nat trans} for a more general statement (whose proof of course does not rely on the present discussion in any way).}  By \cref{nat trans induces equivce betw maps on gpd-complns}, it follows that all the structure maps of $\NerveRezki(\max(\C)) \in s\S$ are equivalences, and hence (since $\bD^{op}$ is sifted so in particular $(\bD^{op})^\gpd \simeq \pt_\S$) it follows that this simplicial space is constant.  The commutativity of the right square now follows from the computation
\[ \NerveRezki(\max(\C))_0 = \left( \Fun([0],\max(\C))^\bW \right)^\gpd \simeq \C^\gpd , \]
which gives rise to a canonical equivalence $\NerveRezki(\max(\C))_\bullet \simeq \const(\C^\gpd) \simeq \Nervei(\C^\gpd)_\bullet$ in $s\S$.
\end{proof}

Now, recall that \textit{any} relative $\infty$-category $(\R,\bW)$ admits a natural map $\min(\R) = (\R,\R^\simeq) \ra (\R,\bW)$ (namely the unit of the adjunction $\min \adj \forget_\Rel$).  Hence, by \cref{rezk nerves of min and max} we obtain a natural map
\[ \Nervei(\R) \ra \NerveRezki(\R,\bW) \]
in $s\S$.\footnote{This can also be obtained from the levelwise inclusion $\hom^\lw_{\Cati}([\bullet],\R) \simeq \left(\Fun^\lw([\bullet],\R)^\bW \right)^\simeq \hookra \Fun^\lw([\bullet],\R)^\bW$ of maximal subgroupoids.}  This immediately suggests the following two questions.

\begin{qn}\label{ask for CSS}
When does this map in $s\S$ (or equivalently, its target) actually lie in the full subcategory $\CSS \subset s\S$?
\end{qn}

\begin{qn}\label{ask for localization up to CSS-replacement}
In light of the composite adjunction
awhat is the $\infty$-categorical significance of this map?
\end{qn}

We give a partial answer to \cref{ask for CSS} in \cite{MIC-hammocks} (see the \textit{calculus theorem} (\Cref{hammocks:calculus result})).  Meanwhile, the essence of the present paper consists in the following complete answer to \cref{ask for localization up to CSS-replacement}, the \bit{local universal property of the Rezk nerve}.

\begin{thm}\label{rezk nerve of a relative infty-category is initial}
For any $(\R,\bW) \in \RelCati$ and any $\C \in \Cati$, we have a commutative square
\[ \begin{tikzcd}
\hom_{\RelCati}((\R,\bW),\min(\C)) \arrow[hook]{r} \arrow{d}[sloped, anchor=north]{\sim} & \hom_{\Cati}(\R,\C) \arrow{d}[sloped, anchor=south]{\sim} \\
\hom_{s\S}(\NerveRezki(\R,\bW),\Nervei(\C)) \arrow{r} & \hom_{\CSS}(\Nervei(\R),\Nervei(\C)) .
\end{tikzcd} \]
In other words, the natural map
\[ \Nervei(\R) \simeq \leftloc_\CSS(\Nervei(\R)) \ra \leftloc_\CSS(\NerveRezki(\R,\bW)) \]
in $\CSS$ corresponds to the localization map $\R \ra \loc{\R}{\bW}$ in $\Cati$.
\end{thm}

\noindent We will give a proof of \cref{rezk nerve of a relative infty-category is initial} in \cref{section proof of nerve}.

Using \cref{rezk nerve of a relative infty-category is initial} as input, we can now prove a statement which will easily imply the \textit{global} universal property of the Rezk nerve (\cref{Rezk nerve computes localization}).

\begin{prop}\label{rezk nerve induces localization}
The composite functor
\[ \RelCati \xra{\NerveRezki} s\S \xra{\leftloc_\CSS} \CSS \simeq \Cati \]
induces an equivalence
\[ \loc{\RelCati}{\bW_\BarKan} \xra{\sim} \Cati . \]
\end{prop}

In the proof of \cref{rezk nerve induces localization}, it will be convenient to have the following terminology.

\begin{defn}\label{define rezk rel str on ssspaces}
We define the subcategory $\bW_\Rezk \subset ss\S$ of \bit{Rezk weak equivalences} to be created by the composite
\[ s(s\S) \xra{s(|{-}|)} s\S \xra{\leftloc_\CSS} \CSS \simeq \Cati . \]
(This name is meant to be suggestive of Rezk's ``complete Segal space'' model structure on the category $ss\Set$ of bisimplicial \textit{sets}.)  We denote the resulting relative $\infty$-category by $ss\S_\Rezk = (ss\S , \bW_\Rezk) \in \RelCati$.  Since left localizations are in particular free localizations (recall \cref{ex left localization qua free localization}), this composite left adjoint induces an equivalence
\[ \loc{ss\S}{\bW_\Rezk} \xra{\sim} \Cati \]
in $\Cati$.
\end{defn}

\begin{proof}[Proof of \cref{rezk nerve induces localization}]
Recalling \cref{pre-nerve as sCSS}, we have a composite adjunction
\[ \begin{tikzcd}[column sep=1.5cm]
ss\S
{\arrow[transform canvas={yshift=0.7ex}]{r}{\Yo_!(\mxm)}[swap, transform canvas={yshift=0.25ex}]{\scriptstyle \bot} \arrow[transform canvas={yshift=-0.7ex}, leftarrow]{r}[swap]{\preNerveRezki}}
& \RelCati
{\arrow[transform canvas={yshift=0.7ex}]{r}{\locL}[swap, transform canvas={yshift=0.25ex}]{\scriptstyle \bot} \arrow[transform canvas={yshift=-0.7ex}, hookleftarrow]{r}[swap]{\min}}
& \Cati .
\end{tikzcd} \]
Moreover, it follows from \cref{rezk nerves of min and max} that the right adjoint of this composite adjunction is precisely that of the composite adjunction
\[ \begin{tikzcd}[column sep=1.5cm]
s(s\S)
{\arrow[transform canvas={yshift=0.7ex}]{r}{s(|{-}|)}[swap, transform canvas={yshift=0.25ex}]{\scriptstyle \bot} \arrow[transform canvas={yshift=-0.7ex}, hookleftarrow]{r}[swap]{s(\const)}}
& s\S
{\arrow[transform canvas={yshift=0.7ex}]{r}{\leftloc_\CSS}[swap, transform canvas={yshift=0.25ex}]{\scriptstyle \bot} \arrow[transform canvas={yshift=-0.7ex}, hookleftarrow]{r}[swap]{\forget_\CSS}}
& \CSS
{\arrow[transform canvas={yshift=0.7ex}]{r}{\Nervei^{-1}}[swap, transform canvas={yshift=0.05ex}]{\sim} \arrow[transform canvas={yshift=-0.7ex}, leftarrow]{r}[swap]{\Nervei}}
& \Cati
\end{tikzcd} \]
whose left adjoint defines $\bW_\Rezk \subset ss\S$, and hence in particular it follows that the right adjoint of our original composite adjunction defines a weak equivalence
\[ ss\S_\Rezk \xla[\approx]{\min \circ \preNerveRezki} \min(\Cati) \]
in $(\RelCati)_\BarKan$.

Next, we claim that the right adjoint $\RelCati \xra{\preNerveRezki} ss\S$ is a relative functor.  To see this, first note that given any $(\R,\bW) \in \RelCati$, we obtain a counit map
\[ (\R,\bW) \we \min(\loc{\R}{\bW}) \]
in $(\RelCati)_\BarKan$ from the adjunction $\locL \adj \min$.  \cref{rezk nerve of a relative infty-category is initial} and \cref{rezk nerves of min and max} then together imply that applying the functor $\RelCati \xra{\preNerveRezki} ss\S$ to this map yields a weak equivalence
\[ \preNerveRezki(\R,\bW) \we \preNerveRezki(\min(\loc{\R}{\bW})) \simeq \const^\lw(\Nervei(\loc{\R}{\bW})) \]
in $ss\S_\Rezk$.  Hence, any weak equivalence $(\R_1,\bW_1) \we (\R_2,\bW_2)$ in $(\RelCati)_\BarKan$ induces a commutative diagram
\[ \begin{tikzcd}
\preNerveRezki(\R_1,\bW_1) \arrow{r} \arrow{d}[sloped, anchor=north]{\approx} & \preNerveRezki(\R_2,\bW_2) \arrow{d}[sloped, anchor=south]{\approx} \\
\const^\lw(\Nervei(\loc{\R_1}{\bW_1})) \arrow{r}[swap]{\sim} & \const^\lw(\Nervei(\loc{\R_2}{\bW_2}))
\end{tikzcd} \]
in $ss\S_\Rezk$, and then the top arrow in this square is also in $\bW_\Rezk \subset ss\S$ since it has the two-out-of-three property.  So this does indeed define a relative functor
\[ (\RelCati)_\BarKan \xra{\preNerveRezki} ss\S_\Rezk . \]

From here, it follows that the right adjoints of our original composite adjunction form a commutative diagram
\[ \begin{tikzcd}
ss\S_\Rezk & & \min(\Cati) \arrow{ll}{\approx}[swap]{\preNerveRezki \circ \min} \arrow{ld}[sloped, pos=0.6]{\min}[sloped, swap, pos=0.35]{\approx} \\
& (\RelCati)_\BarKan \arrow{lu}[sloped, pos=0.2]{\preNerveRezki}
\end{tikzcd} \]
in $(\RelCati)_\BarKan$, and so the entire diagram lies in $\bW_\BarKan \subset \RelCati$ since it has the two-out-of-three property.  Hence, we obtain a commutative diagram
\[ \begin{tikzcd}
ss\S_\Rezk \arrow{r}{s(|{-}|)} \arrow[bend left]{rr}{\approx} & s\S \arrow{r}{\leftloc_\CSS} & \CSS \simeq \Cati \\
(\RelCati)_\BarKan \arrow{u}{\preNerveRezki}[sloped, anchor=north]{\approx} \arrow{ru}[sloped, swap, pos=0.3]{\NerveRezki}
\end{tikzcd} \]
in $(\RelCati)_\BarKan$, which proves the claim.
\end{proof}

\begin{rem}\label{rem local univ property seems non-formal and adjns of rel infty-cats need not be homotopically well-behaved eg KQ rel infty-cats}
It does not appear possible to give a completely hands-off proof of \cref{rezk nerve induces localization}, i.e.\! one not relying on \cref{rezk nerve of a relative infty-category is initial} (or perhaps even one that would prove \cref{rezk nerve of a relative infty-category is initial} as a formal consequence).  More specifically, adjunctions of underlying $\infty$-categories do not necessarily play well with relative $\infty$-category structures, even if one of the adjoints is a relative functor: one must have some control over the behavior of \textit{both} adjoints.

For instance, the geometric realization functor $s\S \xra{|{-}|} \S$ and its restriction to the subcategory $s\Set \subset s\S$ create subcategories of weak equivalences which define the \textit{Kan--Quillen relative $\infty$-category} structures $(s\S , \bW^{s\S}_\KQ) , (s\Set , \bW^{s\Set}_\KQ) \in \RelCati$ (which underlie their respective Kan--Quillen model structures  (see \cref{sspaces:section define kan--quillen model structure})).  Moreover, these relative $\infty$-categories give rise to a diagram
\[ \begin{tikzcd}
s\S
{\arrow[transform canvas={yshift=0.7ex}]{rr}{s(\pi_0)}[swap, transform canvas={yshift=0.25ex}]{\scriptstyle \bot} \arrow[transform canvas={yshift=-0.7ex}, hookleftarrow]{rr}[swap]{s(\disc)}}
\arrow{d}
& & s\Set \arrow{d} \\
\loc{s\S}{(\bW^{s\S}_\KQ)} \arrow{rd}[sloped, swap, pos=0.6]{\sim} & & \loc{s\Set}{(\bW^{s\Set}_\KQ)} \arrow{ld}[sloped, pos=0.6]{\sim} \\
& \S
\end{tikzcd} \]
in which the right adjoint commutes with the respective localization functors: in other words, it induces a weak equivalence
\[ (s\S_\KQ,\bW^{s\S}_\KQ) \lwe (s\Set_\KQ , \bW^{s\Set}_\KQ) \]
in $(\RelCati)_\BarKan$.  Nevertheless, the left adjoint is clearly very far from also defining a weak equivalence in $(\RelCati)_\BarKan$.
\end{rem}

We can now prove the \bit{global universal property of the Rezk nerve}.

\begin{cor}\label{Rezk nerve computes localization}
The composite functor
\[ \RelCati \xra{\NerveRezki} s\S \xra{\leftloc_\CSS} \CSS \xra[\sim]{\Nervei^{-1}} \Cati \]
is canonically equivalent in $\Fun(\RelCati , \Cati)$ to the localization functor
\[ \RelCati \xra{\locL} \Cati . \]
\end{cor}

\begin{proof}
Since these functors both take the subcategory $\bW_\BarKan \subset \RelCati$ into $(\Cati)^\simeq \subset \Cati$, they factor uniquely through the localization
\[ \RelCati \ra \loc{\RelCati}{\bW_\BarKan} . \]
The resulting functors $\loc{\RelCati}{\bW_\BarKan} \ra \Cati$ are then both equivalences, the former by \cref{rezk nerve induces localization} and the latter by \cref{loc induces BK-equivce}.  The result now follows by inspection, using the fact that
\[ \hom_{\Cati}(\Cati,\Cati) \simeq \bbZ/2 \]
(see \cite[Th\'{e}or\`{e}m 6.3]{Toen} or \cite[Theorem 4.4.1]{LurieGoo}).
\end{proof}

\begin{rem}\label{BK proved global univ prop for relcats}
The global universal property of the Rezk nerve (\cref{Rezk nerve computes localization}) can be seen as a generalization of work of Barwick--Kan.  To see this, consider the composite pair of Quillen adjunctions
\[ s(s\Set_\KQ)_\Reedy \adjarr ss\Set_\Rezk \adjarr \strrelcat_\BarKan , \]
where
\begin{itemizesmall}
\item the first is the left Bousfield localization which defines the Rezk model structure (see \cite[Theorem 7.2]{RezkCSS}) and presents the adjunction $\leftloc_\CSS : s\S \adjarr \CSS : \forget_\CSS$, and
\item the second is the Quillen equivalence which defines the Barwick--Kan model structure (see \cite[Theorem 6.1]{BK-relcats}).
\end{itemizesmall}
As the latter is constructed using the lifting theorem for cofibrantly generated model categories, its right adjoint preserves all weak equivalences by definition.  Moreover, Barwick--Kan provide a natural weak equivalence in $s(s\Set_\KQ)_\Reedy$ (and hence also in $ss\Set_\Rezk$) from the Rezk nerve functor to the right adjoint of their Quillen equivalence (see \cite[Lemma 5.4]{BK-relcats}).

Now, consider the commutative triangle
\[ \begin{tikzcd}
s(s\Set_\KQ)_\Reedy \arrow{rd}[swap]{\id_{ss\Set}} & & \strrelcat_\triv \arrow{ld} \arrow{ll}[swap]{\NerveRezk} \\
 & ss\Set_\Rezk
 \end{tikzcd} \]
in $\RelCat$ (in which we take $\strrelcat$ with the \textit{trivial} model structure since we are interested in relative categories themselves here).  Applying the localization functor
\[ \RelCat \hookra \RelCati \xra{\locL} \Cati , \]
this yields a commutative triangle
\[ \begin{tikzcd}
s\S \arrow{rd}[swap]{\leftloc_\CSS} & & \strrelcat \arrow{ll}[swap]{s(|{-}|) \circ \NerveRezk} \arrow{ld}{\Nervei \circ \locL} \\
& \CSS
\end{tikzcd} \]
in $\Cati$, in which
\begin{itemizesmall}
\item the upper map coincides with the composite
\[ \strrelcat \ra \RelCat \hookra \RelCati \xra{\NerveRezki} s\S \]
by \cref{infty-catl rezk nerve agrees with 1-catl rezk nerve}, and
\item the map $\strrelcat \ra \CSS$ can be identified as indicated since by what we have just seen it is equivalent to the projection
\[ \strrelcat \ra \loc{\strrelcat}{\bW_\BarKan} \simeq \Cati \]
to the underlying $\infty$-category (which is indeed given by localization).
\end{itemizesmall}
It follows that we obtain a commutative diagram
\[ \begin{tikzcd}
\strrelcat \arrow{r} \arrow{d}[swap]{\locL} & \RelCati \arrow{r}{\NerveRezki} & s\S \arrow{d}{\leftloc_\CSS} \\
\Cati \arrow{rr}{\sim}[swap]{\Nervei} & & \CSS
\end{tikzcd} \]
in $\Cati$, which is precisely the restriction of the assertion of the global universal property of the Rezk nerve (\cref{Rezk nerve computes localization}) to the category $\strrelcat$, as claimed.
\end{rem}

\begin{rem}
Taken together, \cref{rezk nerve induces localization} and \cref{Rezk nerve computes localization} imply that in fact the adjunction
\[ \begin{tikzcd}[column sep=1.5cm]
ss\S
{\arrow[transform canvas={yshift=0.7ex}]{r}{\Yo_!(\mxm)}[swap, transform canvas={yshift=0.25ex}]{\scriptstyle \bot} \arrow[transform canvas={yshift=-0.7ex}, leftarrow]{r}[swap]{\preNerveRezki}}
& \RelCati
\end{tikzcd} \]
has
\begin{itemizesmall}
\item that both adjoints are relative functors (with respect to their respective Rezk and Barwick--Kan relative structures), and
\item that the unit and counit are both natural weak equivalences.
\end{itemizesmall}
This can be seen as follows.

First of all, recall that in the proof of \cref{rezk nerve induces localization}, we already saw that the right adjoint is a relative functor.  On the other hand, the left adjoint is a relative functor because the composite left adjoint
\[ ss\S \xra{\Yo_!(\mxm)} \RelCati \xra{\locL} \Cati \]
agrees with the left adjoint
\[ ss\S \xra{s(|{-}|)} s\S \xra{\leftloc_\CSS} \CSS \xra[\sim]{\Nervei^{-1}} \Cati \]
(since we have seen in the proof of \cref{rezk nerve induces localization} that they share a right adjoint), and so in fact the subcategory $\bW_\Rezk \subset ss\S$ is created by pulling back the subcategory $\bW_\BarKan \subset \RelCati$.

Next, we can see that the counit map
\[ \Yo_!(\mxm)(\preNerveRezki(\R,\bW)) \ra (\R,\bW) \]
is a weak equivalence in $(\RelCati)_\BarKan$ as follows.  Applying the functor $\RelCati \xra{\locL} \Cati$, we obtain a map
\[ \locL ( \Yo_!(\mxm)(\preNerveRezki(\R,\bW))) \ra \loc{\R}{\bW} \]
in $\Cati$.  Then, again appealing to the fact that these composite left adjoints $ss\S \ra \Cati$ agree, we can reidentify the source as
\[ \locL(\Yo_!(\mxm)(\preNerveRezki(\R,\bW)) \simeq \Nervei^{-1}(\leftloc_\CSS(s(|{-}|)(\preNerveRezki(\R,\bW)))) \simeq \Nervei^{-1}(\leftloc_\CSS(\NerveRezki(\R,\bW))) . \]
So, we can reidentify this map as
\[ \Nervei^{-1}(\leftloc_\CSS(\NerveRezki(\R,\bW))) \ra \loc{\R}{\bW} , \]
which is an equivalence by \cref{rezk nerve of a relative infty-category is initial}.  So the counit map is indeed a weak equivalence in $(\RelCati)_\BarKan$, i.e.\! the counit is a natural weak equivalence.

Finally, we can see that the unit map
\[ \preNerveRezki(\Yo_!(\mxm)(Y)) \ra Y \]
is a weak equivalence in $ss\S_\Rezk$ as follows.  Applying the composite left adjoint
\[ ss\S \xra{\Nervei^{-1} \circ \leftloc_\CSS \circ s(|{-}|)} \Cati \]
and appealing to \cref{Rezk nerve computes localization}, we obtain a map
\[ \locL(\Yo_!(\mxm)(Y)) \ra \Nervei^{-1}(\leftloc_\CSS(s(|{-}|)(Y))) \]
in $\Cati$, and the same equivalence of composite left adjoints $ss\S \ra \Cati$ implies that this is an equivalence.  So the unit map is indeed a weak equivalence in $ss\S_\Rezk$, i.e.\! the unit is a natural weak equivalence.
\end{rem}

\section{The proof of \cref{rezk nerve of a relative infty-category is initial}}\label{section proof of nerve}

Let $(\R,\bW)$ be an arbitrary relative $\infty$-category.  In this section, we show that as a simplicial space, its Rezk nerve $\NerveRezki(\R,\bW)$ enjoys the desired universal property for mapping \textit{into} complete Segal spaces: for any $\C \in \Cati$, we have a commutative diagram
\[ \begin{tikzcd}
\hom_{\RelCati}((\R,\bW),\min(\C)) \arrow[hook]{r} \arrow{d}[sloped, anchor=north]{\sim} & \hom_{\Cati}(\R,\C) \arrow{d}[sloped, anchor=south]{\sim} \\
\hom_{s\S}(\NerveRezki(\R,\bW),\Nervei(\C)) \arrow{r} & \hom_{\CSS}(\Nervei(\R),\Nervei(\C))
\end{tikzcd} \]
in $\S$, as asserted in \cref{rezk nerve of a relative infty-category is initial}.

Most of the proof is reasonably straightforward, and we can give it immediately.  But there will be one technical result (\cref{natural mono}) that is necessary for the proof which will occupy us for the remainder of the section.

\begin{proof}[Proof of \cref{rezk nerve of a relative infty-category is initial}]
By definition, the localization $\loc{\R}{\bW} \in \Cati$ is given as the pushout
\[ \begin{tikzcd}
\bW \arrow{r} \arrow{d} & \R \arrow{d} \\
\bW^\gpd \arrow{r} & \loc{\R}{\bW}
\end{tikzcd} \]
in $\Cati$; under the equivalence $\Nervei : \Cati \xra{\sim} \CSS$, this corresponds to a pushout diagram
\[ \begin{tikzcd}
\Nervei(\bW) \arrow{r} \arrow{d} & \Nervei(\R) \arrow{d} \\
\Nervei(\bW^\gpd) \arrow{r} & \Nervei(\loc{\R}{\bW})
\end{tikzcd} \]
in $\CSS \subset s\S$.  On the other hand, there is an evident commutative diagram
\[ \begin{tikzcd}[row sep=2cm]
(\bW,\bW^\simeq) \arrow{r} \arrow{d} & (\R,\R^\simeq) \arrow{rd} \\
(\bW,\bW) \arrow[crossing over]{rr} \arrow{rrd} & & (\R,\bW) \arrow{d} \\
& & (\loc{\R}{\bW} , \loc{\R}{\bW}^\simeq)
\end{tikzcd} \]
in $\RelCati$.  Applying the functor $\NerveRezki : \RelCati \ra s\S$ and taking the pushout of the upper left span, in light of \cref{rezk nerves of min and max} we obtain a commutative diagram
\[ \begin{tikzcd}[row sep=2cm]
\Nervei(\bW) \arrow{r} \arrow{d} & \Nervei(\R) \arrow{rd} \arrow{d} \\
\Nervei(\bW^\gpd) \arrow{r} \arrow{rrd} & \mathrm{p.o.}^{s\S} \arrow{r} \arrow{rd} & \NerveRezki(\R,\bW) \arrow{d} \\
& & \Nervei(\loc{\R}{\bW})
\end{tikzcd} \]
in $s\S$,
\begin{itemizesmall}
\item where $\mathrm{p.o.}^{s\S}$ denotes the pushout in $s\S$ of the upper left span, and
\item which contains as a subdiagram the above pushout square in $\CSS \subset s\S$.
\end{itemizesmall}
Our goal is to prove that the induced map
\[ \leftloc_\CSS(\NerveRezki(\R,\bW)) \ra \leftloc_\CSS(\Nervei(\loc{\R}{\bW})) \simeq \Nervei(\loc{\R}{\bW}) \]
is an equivalence in $\CSS \subset s\S$.

For notational convenience, let us simply write
\[ \begin{tikzcd}[column sep=1.5cm]
(s\S)^{op} \arrow{r}{\Yo_{(s\S)^{op}}} \arrow[dashed, bend right=20]{rr}[swap]{\Yo_{\CSS^{op}}} & \Fun(s\S,\S) \arrow{r}{- \circ \forget_\CSS} & \Fun(\CSS,\S)
\end{tikzcd} \]
for the restricted contravariant Yoneda functor, so that for any $Y \in s\S$ we have
\[ \Yo_{\CSS^{op}}(Y) = \hom_{s\S}(Y,\forget_\CSS(-)) \simeq \hom_\CSS(\leftloc_\CSS(Y),-) \]
in $\Fun(\CSS,\S)$.  Then, by Yoneda's lemma, our aforestated goal is equivalent to proving that the map
\[ \NerveRezki(\R,\bW) \ra \Nervei(\loc{\R}{\bW}) \]
in $s\S$ induces an equivalence
\[ \Yo_{\CSS^{op}}(\NerveRezki(\R,\bW)) \la \Yo_{\CSS^{op}}(\Nervei(\loc{\R}{\bW})) \]
in $\Fun(\CSS,\S)$.  Moreover, as the functor $s\S \xra{\leftloc_\CSS} \CSS$ commutes with pushouts (being a left adjoint), it follows that the map
\[ \mathrm{p.o.}^{s\S} \ra \Nervei(\loc{\R}{\bW}) \]
in $s\S$ induces an equivalence
\[ \leftloc_\CSS \left( \mathrm{p.o.}^{s\S} \right) \xra{\sim} \leftloc_\CSS(\Nervei(\loc{\R}{\bW})) \simeq \Nervei(\loc{\R}{\bW}) \]
in $\CSS \subset s\S$, and so the above diagram in $s\S$ gives rise to a retraction diagram
\[ \begin{tikzcd}
\Yo_{\CSS^{op}}(\mathrm{p.o.}^{s\S}) & \Yo_{\CSS^{op}}(\NerveRezki(\R,\bW)) \arrow{l} \\
& \Yo_{\CSS^{op}}(\Nervei(\loc{\R}{\bW})) \arrow{lu}[sloped, pos=0.4]{\sim} \arrow{u}
\end{tikzcd} \]
in $\Fun(\CSS,\S)$ into which this map which we must show to be an equivalence fits, and which it therefore suffices to show is in fact a diagram of equivalences.

Now, observe that $\CSS$ is complete and hence in particular admits all cotensors, and observe moreover that the functor
\[ (s\S)^{op} \xra{\Yo_{\CSS^{op}}} \Fun(\CSS,\S) \]
factors through the contravariant Yoneda embedding and hence takes values in functors which commute with cotensors.  So by \cref{commute with cotensors}, it suffices to show that after postcomposition with $\S \xra{\pi_0} \Set$, the above retraction diagram in $\Fun(\CSS,\S)$ becomes a diagram of natural isomorphisms in $\Fun(\CSS,\Set)$.  Hence, it suffices to show that the induced map
\[ \left( \pi_0 \circ \Yo_{\CSS^{op}}(\NerveRezki(\R,\bW)) \right)
\ra \left( \pi_0 \circ \Yo_{\CSS^{op}}(\mathrm{p.o.}^{s\S}) \right) \]
is a natural monomorphism in $\Fun(\CSS,\Set)$.  This follows from the stronger statement that the composite
\[ \left( \pi_0 \circ \Yo_{\CSS^{op}}(\NerveRezki(\R,\bW)) \right)
\ra \left( \pi_0 \circ \Yo_{\CSS^{op}}(\mathrm{p.o.}^{s\S}) \right)
\ra \left( \pi_0 \circ \Yo_{\CSS^{op}}(\Nervei(\R)) \right)
 \]
is a natural monomorphism in $\Fun(\CSS,\Set)$, which in turn follows from \cref{natural mono}.
\end{proof}


We needed the following easy result in the proof of \cref{rezk nerve of a relative infty-category is initial}.

\begin{lem}\label{commute with cotensors}
Let $\C$ be an $\infty$-category admitting cotensors, and suppose we are given two space-valued functors $F,G \in \Fun(\C,\S)$ that commute with cotensors.  Then, a natural transformation $F \ra G$ is a natural equivalence in $\Fun(\C,\S)$ if and only if its postcomposition $\pi_0 F \ra \pi_0 G$ with $\S \xra{\pi_0} \Set$ is a natural isomorphism in $\Fun(\C,\Set)$.
\end{lem}

\begin{proof}
The ``only if'' direction is clear.  So, suppose we are given a natural transformation $F \ra G$ in $\Fun(\C,\S)$ such that the induced natural transformation $\pi_0 F \ra \pi_0 G$ is a natural equivalence in $\Fun(\C,\Set)$.  Since equivalences in $\Fun(\C,\S)$ are determined componentwise, it suffices to show that for any $Y \in \C$, the map $F(Y) \ra G(Y)$ is an equivalence in $\S$.  In turn, since equivalences in $\S$ are created in $\ho(\S)$, by Yoneda's lemma it suffices to show that for any $Z \in \S$, the induced map $[Z,F(Y)]_\S \ra [Z,G(Y)]_\S$ is an isomorphism in $\Set$.  But since $\C$ admits cotensors, then we can reidentify this map via the canonical commutative square
\[ \begin{tikzcd}
\pi_0(F(Z \cotensoring Y)) \arrow{r}{\cong} \arrow{d}[sloped, anchor=north]{\rotatebox{180}{$\cong$}} & \pi_0( G( Z \cotensoring Y)) \arrow{d}[sloped, anchor=south]{\cong} \\
{[Z,F(Y)]_\S} \arrow{r} & {[Z,G(Y)]_\S}
\end{tikzcd} \]
in $\Set$, in which the top arrow is an isomorphism by the assumption that $\pi_0 F \ra \pi_0 G$ is a natural isomorphism and the vertical arrows are isomorphisms by the assumption that $F$ and $G$ commute with cotensors.  This proves the claim.
\end{proof}

Before moving on to \cref{natural mono}, it will be convenient to have the following bit of terminology.

\begin{defn}
A morphism in a model category $\M$ is called a \bit{homotopy epimorphism} if it presents an epimorphism in the underlying $\infty$-category $\loc{\M}{\bW}$.
\end{defn}

We now proceed to the technical heart of the proof of \cref{rezk nerve of a relative infty-category is initial}.  We warn the reader that our proof of the following result is (perhaps unexpectedly, and certainly unsatisfyingly) complicated.

\begin{lem}\label{natural mono}
The map $\Nervei(\R) \ra \leftloc_\CSS(\NerveRezki(\R,\bW))$ is an epimorphism in $\CSS$.
\end{lem}

\begin{proof}
Our proof will proceed using model categories -- primarily $ss\Set_\Rezk$ and $s\Set_\Joyal$, but also a number of others auxiliarily --, and will also use the language of marked simplicial sets (see e.g.\! \sec T.3.1).

We begin by recalling the two Quillen equivalences between $ss\Set_\Rezk$ and $s\Set_\Joyal$ given in \cite{JT}.
\begin{enumerate}
\item\label{evil JT q eq}
 Let us write $\bD^{op} \times \bD^{op} \xra{\pr_2} \bD^{op}$ for the second projection map and $\bD^{op} \xra{i_2} \bD^{op} \times \bD^{op}$ for the functor $\const([0]^\opobj) \times \id_{\bD^{op}}$.  Pullbacks along these two functors induce the Quillen equivalence
\[ \pr_2^* : s\Set_\Joyal \adjarr ss\Set_\Rezk : i_2^* \]
of \cite[Theorem 4.11]{JT}.
\item\label{Rezk nerve JT q eq}
Let us write $(\Delta^i)^\ttgpd \in s\Set$ for the nerve of the strict (i.e.\! objects-preserving) groupoid completion of $[i] \in \strcat$, and let us write $t_! : ss\Set \ra s\Set$ for the left Kan extension
\[ \begin{tikzcd}[column sep=4cm]
\bD \times \bD \arrow{r}{([n],[i]) \mapsto \Delta^n \times (\Delta^i)^\ttgpd} \arrow{d} & s\Set \\
ss\Set \arrow[dashed, bend right=5]{ru}
\end{tikzcd} \]
along the (1-categorical) Yoneda embedding.  This has a right adjoint $t^! : s\Set \ra ss\Set$ given by
\[ t^!(Y) = \{ \{ \hom_{s\Set}(\Delta^n \times (\Delta^i)^\ttgpd , Y) \}_{i \geq 0} \}_{n \geq 0} , \]
and together these fit into the Quillen equivalence
\[ t_! : ss\Set_\Rezk \adjarr s\Set_\Joyal : t^! \]
of \cite[Theorem 4.12]{JT}.
\end{enumerate}

Now, suppose that $\ttR \in s\Set_\Joyal^f$ is a quasicategory presenting $\R \in \Cati$, and let $(\ttR,\ttW) \in s\Set^+$ be the marked simplicial set obtained by marking precisely those edges of $\ttR$ which present maps in $\bW \subset \R$.  For any $n \geq 0$, the $\infty$-category $\Fun([n],\R)$ is presented by the object
\[ \ul{\hom}_{s\Set}(\Delta^n,\ttR) = \{ \hom_{s\Set}(\Delta^n \times \Delta^i , \ttR) \}_{i \geq 0} \in s\Set_\Joyal , \]
and hence its subcategory
\[ \Fun([n],\R)^\bW \subset \Fun([n],\R) \]
is presented by the object
\[ \{ \hom_{s\Set^+}((\Delta^n)^\flat \times (\Delta^i)^\sharp,(\ttR,\ttW)) \}_{i \geq 0} \in s\Set_\Joyal . \]
These constructions are contravariantly functorial in $[n] \in \bD$, and hence we obtain that the Rezk pre-nerve
\[ \preNerveRezki(\R,\bW) = \Fun^\lw([\bullet],\R)^\bW \in s\Cati \]
is presented by the object
\[ \{ \{ \hom_{s\Set^+}((\Delta^n)^\flat \times (\Delta^i)^\sharp,(\ttR,\ttW)) \}_{i \geq 0} \}_{n \geq 0} \in s(s\Set_\Joyal)_\Reedy . \]
From here, we observe that the Quillen adjunction
\[ \id_{ss\Set} : s(s\Set_\Joyal)_\Reedy \adjarr s(s\Set_\KQ)_\Reedy : \id_{ss\Set} \]
presents the left localization adjunction $s((-)^\gpd) : s\Cati \adjarr s\S : s(\forget_\S)$; as all objects of $s(s\Set_\Joyal)_\Reedy$ are cofibrant, it follows that when considered as an object of $s(s\Set_\KQ)_\Reedy$, this same bisimplicial set presents $\NerveRezki(\R,\bW) \in s\S$.  Moreover, in light of the left Bousfield localization
\[ \id_{ss\Set} : s(s\Set_\KQ)_\Reedy \adjarr ss\Set_\Rezk : \id_{ss\Set} \]
presenting the left localization adjunction $\leftloc_\CSS : s\S \adjarr \CSS : \forget_\CSS$, when considered as an object of $ss\Set_\Rezk$, this same bisimplicial set presents the Rezk nerve
\[ \NerveRezki(\R,\bW) = \left( \Fun^\lw([\bullet],\R)^\bW \right)^\gpd \in \CSS . \]
We will denote this bisimplicial set by $\NerveRezk(\ttR,\ttW) \in ss\Set$.\footnote{When $(\ttR,\ttW)\in s\Set^+$ is the ``marked nerve'' of a relative 1-category, this recovers the 1-categorical Rezk nerve of \cref{infty-catl rezk nerve agrees with 1-catl rezk nerve} (as an object of $ss\Set$), and so there is no ambiguity in the notation.}  In particular, note that we have a natural isomorphism $\NerveRezk(\ttR^\natural) \cong t^!(\ttR)$ in $ss\Set$, and hence we see that the right Quillen equivalence
\[ t^! : s\Set_\Joyal \ra ss\Set_\Rezk \]
presents the equivalence $\Nervei : \Cati \xra{\sim} \CSS$ of $\infty$-categories.

Now, the natural map
\[ \ttR^\natural \ra (\ttR,\ttW) \]
in $s\Set^+$ induces a map
\[ \NerveRezk(\ttR^\natural) \ra \NerveRezk(\ttR,\ttW) \]
in $ss\Set_\Rezk$, which by what we have seen presents the map
\[ \Nervei(\R) \ra \leftloc_\CSS(\NerveRezki(\R,\bW)) \]
in $\CSS$.  So, to prove that this latter map is an epimorphism in $\CSS$, it suffices to prove that the former map is a homotopy epimorphism in $ss\Set_\Rezk$.  However, note that there is a natural isomorphism $t_!(\pr_2^*(\ttR)) \xra{\cong} \ttR$ in $s\Set$, which is in particular a weak equivalence in $s\Set_\Joyal$; via the Quillen equivalence of item \ref{Rezk nerve JT q eq}, this corresponds to a weak equivalence $\pr_2^*(\ttR) \we t^!(\ttR)$ in $ss\Set_\Rezk$.  So, it also suffices to show that the composite map
\[ \pr_2^*(\ttR) \we t^!(\ttR) \cong \NerveRezk(\ttR^\natural) \ra \NerveRezk(\ttR,\ttW) \]
is a homotopy epimorphism in $ss\Set_\Rezk$.

For this, let us also recall the ``usual'' geometric realization functor $ss\Set \ra s\Set$ (a homotopy colimit functor with respect to $s(s\Set_\KQ)_\Reedy$): this is the left Kan extension
\[ \begin{tikzcd}[column sep=4cm]
\bD \times \bD \arrow{r}{([n],[i]) \mapsto \Delta^n \times \Delta^i} \arrow{d} & s\Set \\
ss\Set \arrow[dashed, bend right=5]{ru}
\end{tikzcd} \]
along the (1-categorical) Yoneda embedding, but by \cite[Chapter IV, Exercise 1.6]{GJ} this is (naturally isomorphic to) the functor $\diag^* : ss\Set \ra s\Set$, where $\bD^{op} \xra{\diag} \bD^{op} \times \bD^{op}$ denotes the diagonal functor.  Now, the evident morphisms $\Delta^n \times \Delta^i \ra \Delta^n \times (\Delta^i)^\ttgpd$ in $s\Set$ induce a natural transformation $\diag^* \ra t_!$ in $\Fun(ss\Set,s\Set)$.  Moreover, it is not hard to see that upon precomposition with $s\Set \xra{\pr_2^*} ss\Set$, this induces the identity natural transformation from $\id_{s\Set}$ to itself in $\Fun(s\Set,s\Set)$ (up to isomorphism).  Applying these observations to the above composite map in $ss\Set$, we obtain a commutative square
\[ \begin{tikzcd}
\diag^*(\pr_2^*(\ttR)) \arrow{r}{\alpha} \arrow{d}[sloped, anchor=north]{\rotatebox{180}{$\cong$}} & \diag^*(\NerveRezk(\ttR,\ttW)) \arrow{d}{\beta} \\
t_!(\pr_2^*(\ttR)) \arrow{r}[swap]{\gamma} & t_!(\NerveRezk(\ttR,\ttW))
\end{tikzcd} \]
in $s\Set$, where both objects on the left are (compatibly) isomorphic to $\ttR$ itself.  Since $t_! : ss\Set_\Rezk \ra s\Set_\Joyal$ is a left Quillen equivalence and all objects of $ss\Set_\Rezk$ are cofibrant, it suffices to show that the map $\gamma$ is a homotopy epimorphism in $s\Set_\Joyal$.  For this, it suffices to prove that when considered in $s\Set_\Joyal$, the map $\alpha$ is a weak equivalence and the map $\beta$ is a homotopy epimorphism.  This, finally, is what we will show.

We begin with the second assertion, that the map
\[ \diag^*(\NerveRezk(\ttR,\ttW)) \xra{\beta} t_!(\NerveRezk(\ttR,\ttW)) \]
is a homotopy epimorphism in $s\Set_\Joyal$.  In fact, we will show that the natural transformation $\diag^* \ra t_!$ in $\Fun(ss\Set,s\Set_\Joyal)$ is a componentwise homotopy epimorphism.  Just for the duration of this sub-proof, let us ``reverse'' our simplicial coordinates, so that the one we have been denoting by ``$i$'' will be the \textit{outer} coordinate while the one we have been denoting by ``$n$'' will be the \textit{inner} coordinate.  Now, observe that we can rewrite these two functors as
\[ \diag^* \cong \int^{[i] \in \bD} (-)_i \times \Delta^i : s(s\Set) \ra s\Set \]
and
\[ t_! \cong \int^{[i] \in \bD} (-)_i \times (\Delta^i)^\ttgpd : s(s\Set) \ra s\Set , \]
under which identifications our natural transformation $\diag^* \ra t_!$ is induced by the evident map $\Delta^\bullet \ra (\Delta^\bullet)^\ttgpd$ in $c(s\Set)$.  Moreover, by Proposition T.A.2.9.26, we obtain a left Quillen bifunctor
\[ \int^{[i] \in \bD} (-)_i \times (-)^i : s(s\Set_\Joyal)_\Reedy \times c(s\Set_\Joyal)_\Reedy \ra s\Set_\Joyal \]
(since $s\Set_\Joyal$ is cartesian, i.e.\! the product bifunctor is left Quillen).\footnote{Note that since we have flipped our simplicial coordinates, this model structure $s(s\Set_\Joyal)_\Reedy$ is \textit{different} from the model structure $s(s\Set_\Joyal)_\Reedy$ that appeared earlier (with respect to the fixed copy of the underlying category $ss\Set$ in which we have been working).}  As every object of $s(s\Set_\Joyal)_\Reedy$ is cofibrant, for any object
\[ Y_\bullet \in s(s\Set_\Joyal)_\Reedy \]
the above left Quillen bifunctor induces a left Quillen functor
\[ \int^{[i] \in \bD} Y_i \times (-)^i : c(s\Set_\Joyal)_\Reedy \ra s\Set_\Joyal . \]
Moreover, the cofibrant objects of $c(s\Set_\Joyal)_\Reedy$ are exactly those of $c(s\Set_\KQ)_\Reedy$ (since the cofibrations in $s\Set_\Joyal$ are exactly those of $s\Set_\KQ$), and so in particular the objects $\Delta^\bullet,(\Delta^\bullet)^\ttgpd \in c(s\Set_\Joyal)_\Reedy$ are cofibrant by \cite[Corollary 15.9.10]{Hirsch}.

Now, epimorphisms (being determined by a colimit condition) are preserved by left adjoint functors of $\infty$-categories.  Moreover, by \cite[Theorem 2.1]{adjns}, a left Quillen functor between model categories induces a left adjoint functor between $\infty$-categories, which is presented (in $\strrelcat_\BarKan$) by the restriction of the left Quillen functor to the subcategory of cofibrant objects.  So, it suffices to show that the map $\Delta^\bullet \ra (\Delta^\bullet)^\ttgpd$ is a homotopy epimorphism in $c(s\Set_\Joyal)_\Reedy$.

For this, observe that the model category $c(s\Set_\Joyal)_\Reedy$ presents the $\infty$-category $c\Cati$.  Since epimorphisms in $c\Cati = \Fun(\bD,\Cati)$ are determined componentwise, it suffices to show that each $\Delta^i \ra (\Delta^i)^\ttgpd$ is a homotopy epimorphism in $s\Set_\Joyal$.  But this is clear: this map in $s\Set_\Joyal$ presents the terminal map
\[ [i] \ra [i]^\gpd \simeq \pt_\Cati \]
in $\Cati$, which on an arbitrary $\infty$-category $\C$ corepresents the inclusion
\[ \C^\simeq \hookra \hom_\Cati([i],\C) \]
of the subspace of length-$i$ sequences of composable \textit{equivalences} (inside of the space of arbitrary length-$i$ sequences of composable morphisms).  Thus, the natural transformation $\diag^* \ra t_!$ in $\Fun(ss\Set , s\Set_\Joyal)$ is indeed a componentwise homotopy epimorphism, and so in particular we obtain that the map $\beta$ (which is its component at the object $\NerveRezk(\ttR,\ttW) \in ss\Set$) is a homotopy epimorphism, as claimed.

So, it only remains to show that the map
\[ \ttR \cong \diag^*(\pr_2^*(\ttR)) \xra{\alpha} \diag^*(\NerveRezk(\ttR,\ttW)) \]
is a weak equivalence in $s\Set_\Joyal$.  Unwinding the definitions, we see that via the evident cosimplicial object
\[ \bD \xra{(\Delta^\bullet)^\flat \times (\Delta^\bullet)^\sharp} s\Set^+ , \]
we obtain a canonical isomorphism
\[ \diag^*(\NerveRezk(\ttR,\ttW)) \cong \hom_{s\Set^+}^\lw((\Delta^\bullet)^\flat \times (\Delta^\bullet)^\sharp , (\ttR,\ttW)) . \]
Moreover, via the canonical isomorphisms
\[ \ttR \cong \hom^\lw_{s\Set}(\Delta^\bullet , \ttR) \cong \hom^\lw_{s\Set^+}((\Delta^\bullet)^\flat , \ttR^\flat ) \cong \hom^\lw_{s\Set^+}((\Delta^\bullet)^\flat, (\ttR,\ttW)) , \]
this map $\alpha$ is corepresented by the collection of first projection maps
\[ (\Delta^n)^\flat \times (\Delta^n)^\sharp \ra (\Delta^n)^\flat , \]
which assemble to a natural transformation in $\Fun(\bD,s\Set^+)$.  On the other hand, the collection of diagonal maps
\[ (\Delta^n)^\flat \ra (\Delta^n)^\flat \times (\Delta^n)^\sharp \]
(or more precisely, the unique maps in $s\Set^+$ which recover the diagonal maps in $s\Set$ under the forgetful functor $s\Set^+ \ra s\Set$) also assemble into a natural transformation in $\Fun(\bD,s\Set^+)$, which likewise corepresents a map
\[ \diag^*(\NerveRezk(\ttR,\ttW)) \xra{\rho} \ttR \]
in $s\Set$.  Clearly, the composite
\[ \ttR \xra{\alpha} \diag^*(\NerveRezk(\ttR,\ttW)) \xra{\rho} \ttR \]
is the identity map, since this is true of the composite
\[ (\Delta^n)^\flat \ra (\Delta^n)^\flat \times (\Delta^n)^\sharp \ra (\Delta^n)^\flat \]
of the diagonal map followed by the first projection.  On the other hand, we will show that the composite
\[ \diag^*(\NerveRezk(\ttR,\ttW)) \xra{\rho} \ttR \xra{\alpha} \diag^*(\NerveRezk(\ttR,\ttW)) \]
is connected to $\id_{\diag^*(\NerveRezk(\ttR,\ttW))}$ by the zigzag of simplicial homotopies illustrated in \cref{zigzag of simp htpies},
\begin{figure}[h]
\[ \begin{tikzcd}[row sep=1.5cm, column sep=3cm]
\diag^*(\NerveRezk(\ttR,\ttW)) \arrow{d}[swap]{\Delta^{\{0\}} \times \id} \arrow{rdd}{\id_{\diag^*(\NerveRezk(\ttR,\ttW))}} \\
\Delta^1 \times \diag^*(\NerveRezk(\ttR,\ttW)) \arrow[dashed]{rd}{H_1} \\
\diag^*(\NerveRezk(\ttR,\ttW)) \arrow{u}{\Delta^{\{1\}} \times \id} \arrow{d}[swap]{\Delta^{\{1\}} \times \id} \arrow[dashed]{r}{\eta} & \diag^*(\NerveRezk(\ttR,\ttW)) \\
\Delta^1 \times \diag^*(\NerveRezk(\ttR,\ttW)) \arrow[dashed]{ru}{H_2} \\
\diag^*(\NerveRezk(\ttR,\ttW)) \arrow{u}{\Delta^{\{0\}} \times \id} \arrow{ruu}[swap]{\alpha \rho}
\end{tikzcd} \]
\caption{The zigzag of simplicial homotopies in $s\Set$ in the proof of \cref{natural mono}.}\label{zigzag of simp htpies}
\end{figure}
whose components (i.e.\! whose values on the vertices of (the source copies of) $\diag^*(\NerveRezk(\ttR,\ttW))$) are all degenerate edges of (the target copy of) $\diag^*(\NerveRezk(\ttR,\ttW))$.  Postcomposing with an arbitrary fibrant replacement
\[ \diag^*(\NerveRezk(\ttR,\ttW)) \we \bbR ( \diag^*(\NerveRezk(\ttR,\ttW)) ) \fibn \pt_{s\Set} \]
in $s\Set_\Joyal$, we obtain a composite
\[ \Lambda^2_2 \ra \enrhom_{s\Set}(\diag^*(\NerveRezk(\ttR,\ttW)),\diag^*(\NerveRezk(\ttR,\ttW))) \ra \enrhom_{s\Set}(\diag^*(\NerveRezk(\ttR,\ttW)),\bbR(\diag^*(\NerveRezk(\ttR,\ttW)))) \]
in $s\Set_\Joyal$ which, by \cite[Chapter 5, Theorem C]{Joyal-qcats-2} (and \cite[Proposition 4.8]{Joyal-qcats-2} (and the fact that $s\Set_\Joyal$ is cartesian)), presents a zigzag of natural equivalences in $\Cati$ between the functors presented by the maps $\id_{\diag^*(\NerveRezk(\ttR,\ttW))}$ and $\alpha\rho$ in $s\Set_\Joyal$.  In turn, this zigzag (along with the natural equivalence in $\Cati$ presented by the identification $\rho \alpha = \id_\ttR$) witnesses the fact that the maps $\alpha$ and $\rho$ in $s\Set_\Joyal$ present inverse equivalences in $\Cati$, from which we conclude that in particular the map $\alpha$ is indeed a weak equivalence in $s\Set_\Joyal$.

Now, all three of $\eta$, $H_1$, and $H_2$ will be corepresented by maps between the various objects $(\Delta^n)^\flat \times (\Delta^n)^\sharp \in s\Set^+$; in turn, all of these maps will be obtained by applying the evident ``marked nerve'' functor $\Nerve^+ : \strrelcat \ra s\Set^+$ to maps between the various objects $[n] \times [n]_\bW \in \strrelcat$.

We begin by defining the map $\diag^*(\NerveRezk(\ttR,\ttW)) \xra{\eta} \diag^*(\NerveRezk(\ttR,\ttW))$: this is corepresented by the marked nerves of the maps
\[ [n] \times [n]_\bW \xra{\eta^n} [n] \times [n]_\bW \]
in $\strrelcat$ given by
\[ \eta^n(i,j) = \left\{ \begin{array}{ll}
(i,i), & i \geq j \\
(i,j) , & i < j .
\end{array} \right. \]
It is easy to verify that this does indeed define a map in $\strrelcat$, and moreover that assembling these maps for all $n \geq 0$ yields an endomorphism of the object $[\bullet] \times [\bullet]_\bW \in c\strrelcat$.

In order to define the simplicial homotopies $H_1$ and $H_2$, we first recall a combinatorial reformation of the definition of a simplicial homotopy (see e.g.\! \cite[Definitions 5.1]{MaySimp}): for any $Y,Z \in s\Set$ and any $f,g \in \hom_{s\Set}(Y,Z)$, a simplicial homotopy
\[ \begin{tikzcd}[row sep=1.5cm, column sep=3cm]
Y \arrow{d}[swap]{\Delta^{\{0\}} \times \id} \arrow{rd}{g} \\
\Delta^1 \times Y \arrow[dashed]{r}{h} & Z \\
Y \arrow{u}{\Delta^{\{1\}} \times \id} \arrow{ru}[swap]{f}
\end{tikzcd} \]
is equivalently given by a family of maps
\[ \{ h_{i,n} \in \hom_{\Set}(Y_n,Z_{n+1}) \}_{0 \leq i \leq n \geq 0} \]
which satisfy the identities
\[ \delta_0 h_{0,n} = f_n , \]
\[ \delta_{n+1} h_{n,n} = g_n , \]
\[ \delta_i h_{j,n} = \left\{ \begin{array}{ll}
h_{j-1,n-1} \delta_i , & i < j \\
\delta_i h_{i-1,n} , & i = j \not= 0 \\
h_{j,n-1} \delta_{i-1} , & i > j+1 , \end{array} \right. \]
and
\[ \sigma_i h_{j,n} = \left\{ \begin{array}{ll}
h_{j+1,n+1} \sigma_i , & i \leq j \\
h_{j,n+1} \sigma_{i-1} , & i > j . \end{array} \right. \]
So, for $\varepsilon \in \{ 1, 2\}$, we will define the simplicial homotopies
\[ \Delta^1 \times \diag^*(\NerveRezk(\ttR,\ttW)) \xra{H_\varepsilon} \diag^*(\NerveRezk(\ttR,\ttW)) \]
to be corepresented by the marked nerves of families of maps
\[ \{ H^{i,n}_\varepsilon \in \hom_{\strrelcatsup}([n+1] \times [n+1]_\bW , [n] \times [n]_\bW ) \}_{0 \leq i \leq n \geq 0} \]
satisfying the opposites of the identities given above (with the first two ``boundary condition'' identities being dictated by their respective sources and targets).  Namely, we define
\[ H_1^{i,n}(j,k) = \left\{ \begin{array}{ll}
(j,k), & 0 \leq j,k \leq i \\
(j-1,j-1), & j > i \textup{ and } j \geq k \\
(j,k-1), & k > i \geq j \\
(j-1,k-1), & k > j > i
\end{array} \right. \]
and
\[ H_2^{i,n}(j,k) = \left\{ \begin{array}{ll}
(j,j) , & j \leq i \\
(j-1,j-1), & j > i \textup{ and } j \geq k \\
(j-1,k-1), & k > j > i .
\end{array} \right. \]
It is a straightforward (but lengthy) process to verify
\begin{itemize}
\item that these satisfy the opposites of the identities given above,
\item that they restrict along their boundaries to the various maps
\[ \id_{\diag^*(\NerveRezk(\ttR,\ttW))} , \eta , \alpha \rho \in \hom_{s\Set}(\diag^*(\NerveRezk(\ttR,\ttW)) , \diag^*(\NerveRezk(\ttR,\ttW)) ) \]
as indicated in \cref{zigzag of simp htpies}, and
\item that their values on vertices are all degenerate edges,
\end{itemize}
as claimed.  This completes the proof.
\end{proof}

\bibliographystyle{amsalpha}
\bibliography{rnerves}{}

\providecommand{\bysame}{\leavevmode\hbox to3em{\hrulefill}\thinspace}
\providecommand{\MR}{\relax\ifhmode\unskip\space\fi MR }
\providecommand{\MRhref}[2]{%
  \href{http://www.ams.org/mathscinet-getitem?mr=#1}{#2}
}
\providecommand{\href}[2]{#2}
\begin{thebibliography}{LMG15}

\bibitem[BK12]{BK-relcats}
C.~Barwick and D.~M. Kan, \emph{Relative categories: another model for the
  homotopy theory of homotopy theories}, Indag. Math. (N.S.) \textbf{23}
  (2012), no.~1-2, 42--68.

\bibitem[GJ99]{GJ}
Paul~G. Goerss and John~F. Jardine, \emph{Simplicial homotopy theory}, Progress
  in Mathematics, vol. 174, Birkh\"auser Verlag, Basel, 1999.

\bibitem[Hir03]{Hirsch}
Philip~S. Hirschhorn, \emph{Model categories and their localizations},
  Mathematical Surveys and Monographs, vol.~99, American Mathematical Society,
  Providence, RI, 2003.

\bibitem[Joy]{Joyal-qcats-2}
Andr{\'{e}} Joyal, \emph{{The Theory of Quasi-Categories and its Applications,
  Volume II}}, available at {\tt
  http://mat.uab.cat/{$\thicksim$kock}/crm/hocat/advanced-course/Quadern45-2.pdf}.

\bibitem[JT07]{JT}
Andr{\'e} Joyal and Myles Tierney, \emph{Quasi-categories vs {S}egal spaces},
  Categories in algebra, geometry and mathematical physics, Contemp. Math.,
  vol. 431, Amer. Math. Soc., Providence, RI, 2007, pp.~277--326.

\bibitem[LMG15]{LowMG}
Zhen~Lin Low and Aaron Mazel-Gee, \emph{{From fractions to complete Segal
  spaces}}, Homology Homotopy Appl. \textbf{17} (2015), no.~1, 321--338.

\bibitem[Lur09a]{LurieHTT}
Jacob Lurie, \emph{Higher topos theory}, Annals of Mathematics Studies, vol.
  170, Princeton University Press, Princeton, NJ, 2009, also available at {\tt
  http://math.harvard.edu/{$\thicksim$lurie}} and at {\tt arXiv:math/0608040},
  v4.

\bibitem[Lur09b]{LurieGoo}
\bysame, \emph{{$(\infty,2)$-categories and the Goodwillie calculus I}},
  available at {\tt http://www.math.harvard.edu/{$\thicksim$lurie}}, version
  dated October 8, 2009.

\bibitem[Lur14]{LurieHA}
\bysame, \emph{Higher algebra}, available at {\tt
  http://www.math.harvard.edu/{$\thicksim$lurie}}, version dated September 14,
  2014.

\bibitem[May92]{MaySimp}
J.~Peter May, \emph{Simplicial objects in algebraic topology}, Chicago Lectures
  in Mathematics, University of Chicago Press, Chicago, IL, 1992, Reprint of
  the 1967 original.

\bibitem[MGa]{MIC-sspaces}
\emph{{\textup{{A}aron {M}azel-{G}ee,} {\sspacestitle}}}, available at {\tt
  arXiv:1412.8411}, v2.

\bibitem[MGc]{MIC-gr}
\emph{{\bysame, {\grtitle}}}, to appear.

\bibitem[MGd]{MIC-hammocks}
\emph{{\bysame, {\hammockstitle}}}, to appear.

\bibitem[MGe]{MIC-qadjns}
\emph{{\bysame, {\qadjnstitle}}}, to appear.

\bibitem[MGf]{MIC-fundthm}
\emph{{\bysame, {\fundthmtitle}}}, to appear.

\bibitem[MGq]{adjns}
\emph{{\bysame, {Quillen adjunctions induce adjunctions of quasicategories}}},
  available at {\tt arXiv:1501.03146}, v1.

\bibitem[Rez01]{RezkCSS}
Charles Rezk, \emph{A model for the homotopy theory of homotopy theory}, Trans.
  Amer. Math. Soc. \textbf{353} (2001), no.~3, 973--1007 (electronic).

\bibitem[To{\"e}05]{Toen}
Bertrand To{\"e}n, \emph{Vers une axiomatisation de la th\'eorie des
  cat\'egories sup\'erieures}, $K$-Theory \textbf{34} (2005), no.~3, 233--263.

\end{thebibliography}

\end{document}